%% file: main.tex
\documentclass[11pt]{amsart}

\usepackage{amsfonts, amsmath, amssymb, amsthm, amscd}
\usepackage{mathtools}
\usepackage{graphicx}
\usepackage{color, xcolor}
	\definecolor{Blue}{HTML}{3d25b9}
	\definecolor{Red}{HTML}{c00054} 
\usepackage{booktabs}
\usepackage{tabularx}
\usepackage{blkarray}
\usepackage{makecell} 
\usepackage{multicol}

\usepackage[margin=25mm]{geometry}

\usepackage{marginnote}
\long\def\@savemarbox#1#2{\global\setbox#1\vtop{\hsize\marginparwidth 
  \@parboxrestore\tiny\raggedright #2}}
 \normalmarginpar

\usepackage[colorlinks=true,linkcolor=Blue,citecolor=Blue]{hyperref}
	\hypersetup{breaklinks=true}
\usepackage[capitalise,noabbrev]{cleveref}
	\crefname{equation}{equation}{equations}
\newcommand{\order}[1]{}

\theoremstyle{plain}
	\newtheorem{theorem}{Theorem}
	
	\newtheorem{corollary}[theorem]{Corollary}
	\newtheorem{lemma}[theorem]{Lemma}
	
	\numberwithin{theorem}{section}
\theoremstyle{definition}
	\newtheorem{definition}[theorem]{Definition}
	
\theoremstyle{remark}
	\newtheorem {remark }{Remark}
	\newtheorem{remark}[theorem]{Remark}

\begin{document}
\title{Twisting, ladder graphs and {A}-polynomials}
\author{Em K. Thompson}
\address{School of Mathematics,
Monash University,
VIC 3800, Australia}
\email{em.thompson@monash.edu}

\begin{abstract}
    We extend recent work by Howie, Mathews and Purcell to simplify the calculation of A-polynomials for any family of hyperbolic knots related by twisting. The main result follows from the observation that equations defining the deformation variety that correspond to the twisting are reminiscent of exchange relations in a cluster algebra. We prove two additional results with analogues in the context of cluster algebras: the Laurent phenomenon, and intersection numbers appearing as exponents in the denominator. We demonstrate our results on the twist knots, and on a family of twisted torus knots for which A-polynomials have not previously been calculated.
\end{abstract}

\maketitle

\section{Introduction}
The A-polynomial is an invariant of a (framed) one-cusped 3-manifold that was originally introduced in 1994~\cite{CCGLS}. It is a 2-variable polynomial in $L$ and $M$, describing the relationship between the eigenvalues of the meridian and longitude of the cusp under representations of the fundamental group into SL$(2,\mathbb{C})$. This polynomial carries a number of important properties including the ability to detect boundary slopes of incompressible surfaces in the knot complement. The A-polynomial is also known in connection to the coloured Jones polynomial through the so-called AJ conjecture~\cite{Ga04}. Unfortunately, the A-polynomial is difficult to compute in general, and effective methods of computation remain elusive. In a recent paper by Howie, Mathews and Purcell~\cite{HMP}, equations involved in the calculations of A-polynomials were shown to resemble exchange relations of a cluster algebra. In this paper, we make use of this rich algebraic structure to simplify the calculations of A-polynomials for infinite families of knots related by twisting. 

\subsection{The A-polynomial}\label{sec:Apoly} To compute the A-polynomial of a knot naively, one can assign arbitrary SL$(2,\mathbb{C})$ matrices to each generator in the knot group and set up a system of equations that ensure the group relations are satisfied. Considering the words corresponding to the meridian and longitude, and declaring the eigenvalues of their images to be $M$ and $L$, respectively, we obtain further equations involving these variables.
The number of equations in this system scales linearly with the number of relations in the fundamental group. Eliminating all variables other than $M$ and $L$ gives the A-polynomial. This approach is effective so long as the number of relations is small, and as such, the A-polynomial is readily computable for knots with small crossing number (using, for instance, the Wirtinger presentation, which requires one less relation than there are crossings in a diagram). The A-polynomial has also been calculated for some infinite families of knots with simple fundamental group presentations, such as the twist knots: a recursive formula was given by Hoste and Shanahan~\cite{HosteShanahan04} (recovered and generalised later by Petersen~\cite{Pe15}), and then made explicit by Mathews~\cite{Mathews2014, Mathews2014err}. It is a well known problem in elimination theory that finding resultants in a system of polynomial equations becomes computationally difficult when the number of equations is large or the degree of the polynomials is high. This is a recurring challenge in the calculation of A-polynomials, which we partially address in this paper. 

Based on Thurston's study of the deformation variety of hyperbolic knots~\cite{Thurston:3DGT}, Champanerkar~\cite{Champanerkar:Thesis} developed a method for computing an analogue of the SL$(2,\mathbb{C})$ A-polynomial. He showed that this method results in a polynomial that is a divisor of a PSL$(2,\mathbb{C})$ version of the A-polynomial, which is explicitly related to the SL$(2,\mathbb{C})$ A-polynomial. In particular, Champanerkar's polynomial is guaranteed to include a factor that corresponds to a complete hyperbolic structure on the knot complement~\cite{Champanerkar:Thesis}. This factor is equal to a corresponding factor in the PSL$(2,\mathbb{C})$ A-polynomial containing a discrete, faithful representation associated with the complete structure. As such, we call this factor of the PSL$(2,\mathbb{C})$ A-polynomial, or the corresponding factor of the SL$(2,\mathbb{C})$ A-polynomial, the \textit{geometric factor}. Champanerkar showed that his polynomial detects boundary slopes of incompressible surfaces in the knot complement, in the same way that the SL$(2,\mathbb{C})$ A-polynomial of~\cite{CCGLS} does. 

Champanerkar's polynomial can be calculated directly for hyperbolic knots that are built from a small number of tetrahedra; however, once the number of tetrahedra required to triangulate the knot complement becomes too large, calculations are again impeded by the limitations of elimination theory. Culler developed a numerical method for computing divisors of the SL$(2,\mathbb{C})$ A-polynomial that contain the geometric factor, which also uses the deformation variety. He set up a database of these polynomials for knots with small crossing numbers and knots with low triangulation complexity~\cite{Culler}. 

To date, A-polynomials, or divisors containing the geometric factor, are known for all knots with up to eight crossings, many knots with nine crossings and some knots with ten crossings, as well as all hyperbolic knots that can be triangulated by up to seven ideal tetrahedra~\cite{Culler}. There also exist explicit formulas for the A-polynomials of the torus knots~\cite{CCGLS}, the twist knots~\cite{Mathews2014,Mathews2014err}, iterated torus knots~\cite{NiZh17}, and knots with Conway's notation $C(2m,3)$~\cite{HL16}. Recursive formulas exist for the A-polynomials of certain classes of two-bridge knots~\cite{HosteShanahan04,Pe15}, and a family of pretzel knots~\cite{GM11,TY04}. 

This family of pretzel knots is found by $1/m$ Dehn fillings of what Garoufalidis calls a \textit{favorable link}~\cite{Ga14}. That is, the geometric factors of A-polynomials for the $1/m$ fillings of this link satisfy a particular recurrence. Indeed, Garoufalidis proves more generally that there exists a recurrent sequence of rational functions containing the geometric factor of the A-polynomial for any family of knots related by twisting (see Theorem 3.1. of~\cite{Ga14}). Our results lead to a similar observation but where the rational functions are given explicitly rather than recursively.

In Section~\ref{sec:WHsis} we add to the list of known A-polynomials by giving explicit formulas for rational functions that contain the geometric factor of the A-polynomials for the twisted torus knots $T(5,-5n-14,2,2)$ and $T(5,5n+11,2,2)$ for $n\geq 1$.
Indeed, our results apply more broadly than this. Our main theorem, stated generally below, applies to any family of knots related by twisting. In fact, our methods also apply to one-cusped manifolds more generally, but we restrict our focus to knots in the 3-sphere. 

\begin{theorem}\label{thm:Apoly_general}
    Let $K_{\pm m}$ be the sequence of knots obtained by performing $\pm 1/m$ Dehn fillings on an unknotted component of a two-component link in $S^3$. Then, for sufficiently large $m$, the {A-polynomial} of  $K_{\pm m}$ may be defined by a finite number of fixed polynomial equations corresponding to the parent link and a single polynomial equation depending on $m$ that corresponds to the Dehn filling.
\end{theorem} 
This is stated precisely in Corollary~\ref{cor:Apoly_precise}.

\subsection{Connections to cluster algebras}\label{sec:Cluster} 

Howie, Mathews and Purcell~\cite{HMP} performed a change of basis on the equations used in Champanerkar's method~\cite{Champanerkar:Thesis} that is similar to work of Dimofte~\cite{DimofteQRCS}. When they studied the resulting equations in the context of knots related by Dehn filling, they observed that the equations corresponding to the Dehn filling are reminiscent of \textit{exchange relations} in a cluster algebra. A \textit{cluster algebra} is a commutative ring for which generators and relations are not defined at the outset. Instead, \textit{cluster variables} are defined inductively using a process called \textit{mutation}. Cluster variables belong to sets called \textit{clusters} and any two overlapping clusters are related by an exchange relation that replaces one cluster variable with a new one. 

Cluster algebras were first defined by Fomin and Zelevinsky in the early 2000s when they were studying dual canonical bases and total positivity in semisimple Lie groups~\cite{FZ:partone}. Since then, applications of cluster algebras have been found in a wide range of contexts, including quiver representations, discrete  dynamical systems, tropical geometry, and Teichm\"uller theory~\cite{FZ:conference}. One intriguing property of a cluster algebra, known as the \textit{Laurent phenomenon}, is that every cluster variable can be written as an integer Laurent polynomial in the initial cluster variables~\cite{FZ:partone}. 

A cluster algebra may be of either \textit{finite} or \textit{infinite type}, depending on the number of clusters they contain. The simplest cluster algebra of infinite type~\cite{CZ06} can be defined using the initial cluster $\{ x_1, x_2\}$ and the exchange relation 
\[ x_{i-1} x_{i+1}= x_{i}^2+1. \]  
The equations of Howie, Mathews and Purcell are comparable to this exchange relation, where we instead use variables $\gamma_s$ corresponding to edge classes in the triangulation, indexed by their slope $s$. With this comparison in mind we may exploit what is known about cluster algebras. In particular, we may adapt a formula that exists for all of the cluster variables in the simplest cluster algebra of infinite type. There are three distinct proofs of this formula, given by Caldero and Zelevinsky~\cite{CZ06}, Musiker and Propp~\cite{MP06}, and Zelevinsky~\cite{Zel07}. We use arguments similar to Musiker and Propp to prove the following result. 

\begin{theorem}\label{thm:laurentgen}
    The single polynomial equation of Theorem~\ref{thm:Apoly_general} corresponding to the Dehn filling can be used to express the variable $\gamma_{h}$ as an integer Laurent polynomial in the variables $\gamma_{f},\gamma_{o},\gamma_{p}$.
\end{theorem}
This is stated precisely in Theorem~\ref{thm:laurent_precise}. Note that $f,h,o,p$ are specific slopes that will be defined in due course. 

In the context of cluster algebras associated with triangulations of surfaces, Fomin, Shapiro and Thurston proved that the cluster variables carry information about certain intersection numbers~\cite{FominShapiroThurston08}. In particular, the exponents of the terms in the denominator of the Laurent polynomial are equal to intersection numbers in the corresponding triangulation (see Theorem 8.6 in~\cite{FominShapiroThurston08} for details). We show that a similar result applies in our context with the intersection numbers arising from the Farey triangulation. 
To state the following result we use the fact that each cluster variable can be associated to a rational number (or infinity) and hence to an ideal vertex in the Farey triangulation. 
\begin{theorem}
    Let $\alpha_s$ be a geodesic in $\mathbb{H}^2$ with endpoints labelled by the slopes $h$ and $s$. The exponent of $\gamma_{s}$ in the denominator of the Laurent polynomial for $\gamma_h$ (as in Theorem~\ref{thm:laurentgen}) is equal to the intersection number of $\alpha_s$ with edges in the Farey triangulation of $\mathbb{H}^2$.
\end{theorem}
This is stated precisely in Theorem~\ref{thm:intnos_precise}.

\subsection{Structure of this paper}
In Section~\ref{sec:background} we outline some relevant background, first summarising the work of Howie, Mathews and Purcell, then presenting definitions and results from combinatorics that play a role in our main proofs. Precise statements of our results are given in Section~\ref{sec:results}, along with their proofs, which rely heavily on perfect matchings of appropriately weighted ladder graphs. We present the results that have connections to cluster algebras first, then apply these to the context of A-polynomial calculations. We end in Section~\ref{sec:examples} with examples of how our method can be used to explicitly compute A-polynomials for two families of knots related by twisting: the twisted torus knots $T(5,1-5n,2,2)$, and the twist knots $J(2,2n)$.

\subsection{Acknowledgements}
This research was supported by an Australian Government Research Training Program (RTP) Scholarship. The author thanks Jessica Purcell and Daniel Mathews for their support and guidance. The author is also very grateful to Josh Howie, Stephan Tillmann and Norm Do for giving valuable feedback on a draft of the paper, and to the referee for their helpful comments that greatly improved its exposition.

\section{Background}\label{sec:background}
In this section we review the method for calculating A-polynomials described by Howie, Mathews and Purcell in~\cite{HMP}, including the construction of a layered solid torus and its relationship to the Farey triangulation. We also summarise some relevant combinatorial concepts that appear in later proofs.

\subsection{The A-polynomial from Ptolemy equations}\label{sec:Ptolemy} 
In Champanerkar's work~\cite{Champanerkar:Thesis}, the A-polynomial is defined by the set of gluing equations and cusp equations for an ideal triangulation of a knot complement. This information can be stored in the \textit{Neumann-Zagier (NZ) matrix}~\cite{NeumannZagier}. Neumann and Zagier showed that this matrix exhibits symplectic properties~\cite{NeumannZagier} and in 2013, Dimofte~\cite{DimofteQRCS} used this symplectic structure to perform a change of basis. The result of this is a set of equations, one per tetrahedron, that defines the deformation variety. 

Howie, Mathews and Purcell~\cite{HMP} analysed the equations resulting from Dimofte's change of basis and observed Ptolemy-like structure similar to the equations defining Goerner and Zickert's \textit{enhanced Ptolemy variety}~\cite{GoernerZickert}. In addition, they observed that the equations corresponding to Dehn fillings were particularly simple and were reminiscent of the exchange relations in a cluster algebra (see Section~\ref{sec:Cluster}).

\subsubsection{Layered solid tori and the Farey triangulation}

Howie, Mathews and Purcell were particularly interested in the behaviour of the Ptolemy-like equations corresponding to Dehn fillings. To perform Dehn fillings on triangulated link complements they used \textit{layered solid tori}. Layered solid tori were originally introduced by Jaco and Rubinstein in~\cite{JacoRubinstein:LST} but the construction used here more closely resembles the work of Guerit\'aud and Schleimer~\cite{GueritaudSchleimer}. 

To Dehn fill one cusp of a two-component link complement using a layered solid torus, the link complement must have an ideal triangulation in which only two ideal vertices from two distinct tetrahedra meet the cusp to be filled. Howie, Mathews and Purcell show that this is always possible in Proposition 5.1 of~\cite{HMP}. Given such a triangulation, we may remove the two tetrahedra meeting the cusp, leaving a once-punctured torus boundary component. We glue the layered solid torus to this once-punctured torus boundary. Figure~\ref{fig:schematic} shows this process schematically.

\begin{figure}
    \centering
    \includegraphics[width=0.8\textwidth]{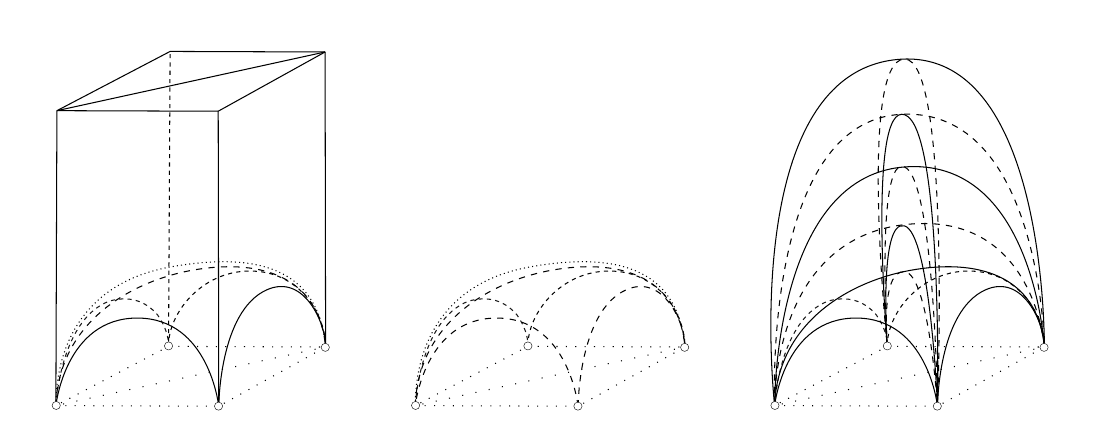}
    \caption{Left: Only two tetrahedra meet the cusp to be filled (located at the point at infinity). Centre: Removing the two tetrahedra leaves a once-punctured torus boundary. Right: Tetrahedra are layered onto the once-punctured torus boundary. Then a fold across an edge (not shown) closes the layered solid torus, thus performing the Dehn filling.}
    \label{fig:schematic}
\end{figure}

To begin constructing the layered solid torus, we glue two adjacent faces of an ideal tetrahedron to the once-punctured torus boundary. Note that this does not change the topology of the link complement but it does introduce a new once-punctured torus boundary with a different triangulation. The new boundary triangulation shares two edges with the previous one, while the third edge is flipped (see Figure~\ref{fig:diagexch}). This is referred to as a \textit{diagonal exchange}.

\begin{figure}
    \centering
    \includegraphics[width=0.5\textwidth]{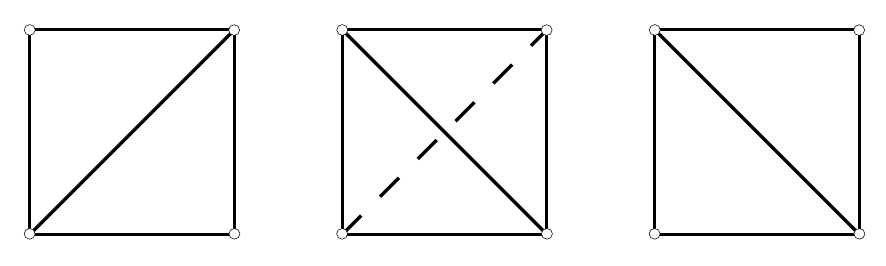}
    \caption{Left: the triangulation on the original boundary. Centre: the ideal tetrahedron glued to the boundary. Right: the triangulation on the new boundary.}
    \label{fig:diagexch}
\end{figure}

We continue layering ideal tetrahedra onto the boundary until the desired boundary triangulation is obtained.\footnote{It is possible to define degenerate layered solid tori consisting of either no tetrahedra or one tetrahedron but we will not need these constructions here. Descriptions of these can be found in~\cite{HMP}.} At this point, the tetrahedra we have introduced form a complex that is homotopy equivalent to a thickened once-punctured torus. To form a solid torus we close up the inner-most layer by identifying the two exposed ideal triangles. This can be seen as folding across one of the exposed edges. The tetrahedra that have been introduced now form a solid torus in which a particular edge is homotopically trivial.  

Importantly, this construction allows boundary curves with any rational slope to be made homotopically trivial. The original boundary triangulation consists of three ideal edges, each with a well-defined slope in terms of the meridian and longitude of the torus boundary. As a tetrahedron is added, the diagonal exchange introduces a new edge with a different slope. However, there are only three possible slopes that the new edge may have, depending on which edge is covered by the diagonal exchange. This behaviour is well-understood and is captured by the structure of the Farey triangulation.

The \textit{Farey triangulation} is an ideal triangulation of $\mathbb{H}^2$, with edges connecting vertices labelled by rational slopes $a/b$ and $c/d$ whenever $|ad-bc|=1$ (see Figure~\ref{fig:Farey}). Since one-vertex triangulations of the torus consist of three edges whose pairwise intersection number is one, each triangle in the Farey triangulation corresponds to a triangulation of the once-punctured torus (for more on this correspondence see, for example, Section 3.1 of~\cite{FutGue06}). In particular, the boundary triangulations seen during the construction of a layered solid torus each correspond to a triangle in the Farey triangulation. Moreover, since consecutive boundary triangulations only differ by a diagonal exchange, they appear as adjacent triangles in the Farey triangulation. As a result, we may use a walk in the Farey triangulation to encode the construction of a layered solid torus. 

\begin{figure}
    \centering
    \includegraphics[width=0.45\textwidth]{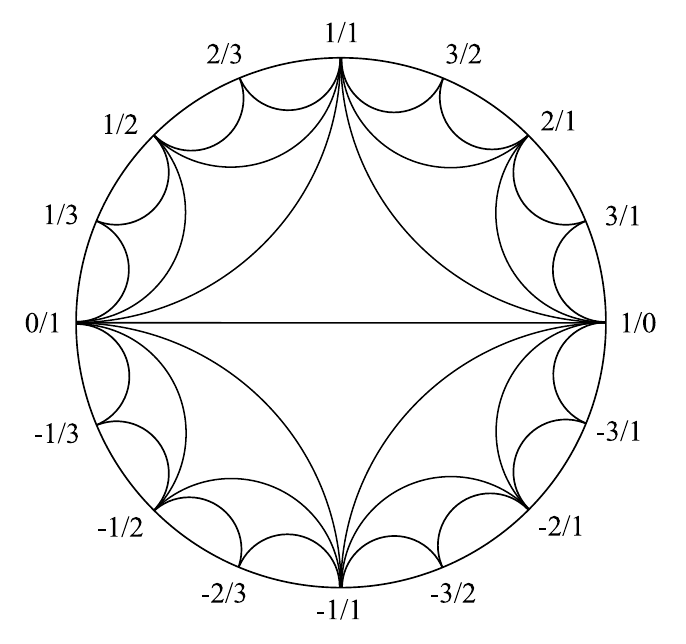}
    \caption{The Farey triangulation of $\mathbb{H}^2$ using the Poincar\'e disk model.}
    \label{fig:Farey}
\end{figure}

A walk in the Farey triangulation passes through a sequence of triangles. We label these triangles $T_0, T_1,\dots, T_{N+1}$ and refer to the step between $T_k$ and $T_{k+1}$ as the $k^{th}$ step. In the construction of a layered solid torus we never perform a diagonal exchange on an edge that was introduced by the previous layer.\footnote{An astute reader may have noticed that the complex in the right of Figure~\ref{fig:schematic} disobeys this rule!} This rule ensures that the corresponding walk in the Farey triangulation contains no backwards steps. Therefore, once $T_0$ and $T_1$ have been identified, all subsequent steps may be viewed as either a left step or a right step. As such, the construction of a layered solid torus can be completely described by the initial information $T_0,T_1$ along with a sequence of left and right steps. Note that the step from $T_N$ to $T_{N+1}$ corresponds to the folding that closes the layered solid torus, rather than the addition of a new tetrahedron. 

\begin{definition}[Anatomy of a layered solid torus]\label{def:tail}
    Let $W$ be a word in L's and R's describing the sequence of left and right steps in the construction of a layered solid torus $X$. 
    \begin{itemize}
        \item The final letter in $W$, corresponding to the fold in $X$, is the \textit{tip of} $W$.
        \item The maximal string of either L's or R's immediately preceding the tip of $W$ is the \textit{tail of} $W$ and the corresponding tetrahedra form the \textit{tail of} $X$.
        \item The string of L's and R's in $W$ preceding the tail of $W$ form the \textit{body of} $W$ and the corresponding tetrahedra form the \textit{body of} $X$.
        \item The tetrahedron in $X$ that corresponds to the $0^{th}$ step is the \textit{head of} $X$.
    \end{itemize}
\end{definition}

\begin{remark}
    When referring to the length of a walk that describes the construction of a layered solid torus (that is, including the head, body, tail and tip) we use $N$, whereas when only considering the length of a tail we use $n$.
\end{remark}

\subsubsection{Ptolemy equations corresponding to a layered solid torus}

Let us now establish notation for the slopes in a layered solid torus with reference to the corresponding walk in the Farey triangulation.  Our notation differs to that used in~\cite{HMP}, where slopes are labelled according to the \textit{absolute} direction of the associated step (that is, using \textit{port} for the slope to the left and \textit{starboard} for the slope to the right). Here we label slopes according to the direction of the associated step \textit{relative} to the previous step. For the $k^{th}$ step, we label the \textit{old} slope $o_k$ and the slope we are \textit{heading} towards $h_k$, as in~\cite{HMP}. Knowing the $(k-1)^{st}$ step, we label the slope that the $k^{th}$ step \textit{pivots} around $p_k$ and the slope that \textit{fans} out $f_k$ (as in Figure~\ref{fig:slopelabels}, right). For the initial step, the old and heading slopes are labelled $o_0$ and $h_0$, respectively. However, because there is no previous step, the pivot and fan slopes are ill-defined. Hence, for this step we declare the slope to the left in the Farey triangulation to be $f_0$ and the slope to the right to be $p_0$ (see Figure~\ref{fig:slopelabels}, left). Note that the labelling of the initial step is as though it were a right step.

\begin{figure}
    \centering
    \includegraphics[width=0.9\textwidth]{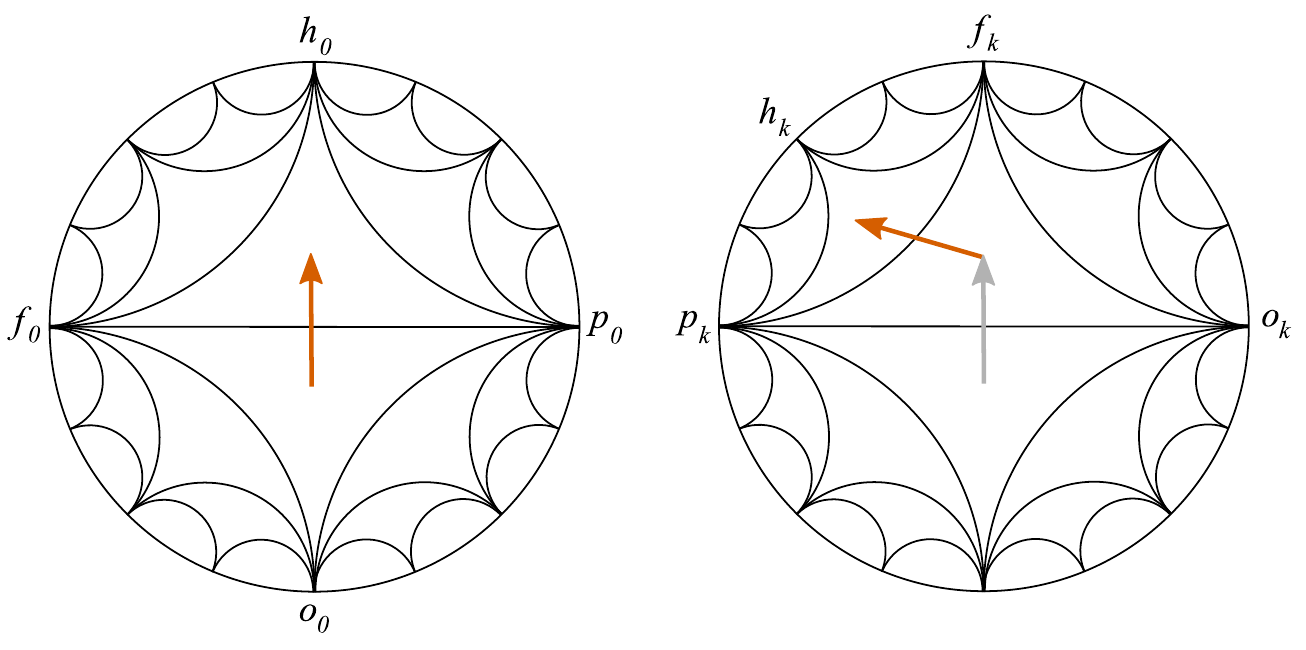}
    \caption{Left: Slope labels for the initial step. Right: Slope labels for the $k^{th}$ step.}
    \label{fig:slopelabels}
\end{figure}

Following Howie, Mathews and Purcell, we assign $\gamma$ variables to each edge class in the triangulation and label these variables by the slope of the edge.  There are two formats we use, depending on the context. When referring to a slope $s_k$ associated to the $k^{th}$ step in the construction of a layered solid torus, we use the notation $\gamma_{s_k}$. When the actual slope is known, as is the case throughout Section~\ref{sec:examples}, we use the notation $\gamma_{p/q}$ for the edge with slope $p/q$.

Here we restate Theorem 3.17(ii) of~\cite{HMP} using the relative labelling of slopes discussed above. 

\begin{theorem}[Howie, Mathews \& Purcell, Theorem 3.17(ii) of~\cite{HMP}]\label{thm:HMP}
With slopes labelled according to the corresponding walk of length $N$ in the Farey triangulation, the Ptolemy equations for the tetrahedra in a layered solid torus are 
\[ \gamma_{o_k}\gamma_{h_k}+\gamma_{p_k}^2-\gamma_{f_k}^2=0, \text{ for } 0\leq k\leq N-1. \]
\end{theorem}
When $k=N$ we pick up the \textit{folding equation} $\gamma_{p_N}=\gamma_{f_N}$.

\begin{remark}
    By labelling slopes according to their relative direction, we remove the need to distinguish between left and right steps (as in~\cite{HMP}). Observe that the Ptolemy equations for ${0\leq k\leq N-1}$ encompass those associated with the head, body and tail of the layered solid torus, while the folding equation corresponds to the tip.
\end{remark}

\subsection{Combinatorial tools}\label{sec:Comb} 

In this section we recall definitions and results from combinatorics that will be used in Section~\ref{sec:results}.
 A \textit{ladder graph} $L_r$, informally, is the graph that resembles a ladder with $r$ rungs. More formally, it is a graph on $2r$ vertices arranged in two rows of $r$ vertices, with edges connecting adjacent vertices in each row and column. A \textit{weighted graph} is a graph in which each edge is assigned a number or variable, called a \textit{weight}. A \textit{perfect matching} of a graph $G$ is a subset $S$ of edges in $G$ such that each vertex belongs to exactly one edge in $S$ (see Figure~\ref{fig:perfectmatching} for an example). The \textit{weight} $w(S)$ of a perfect matching is defined to be the product of the weights of its constituent edges. 

\begin{figure}[ht]
    \centering
    \includegraphics[width=0.5\textwidth]{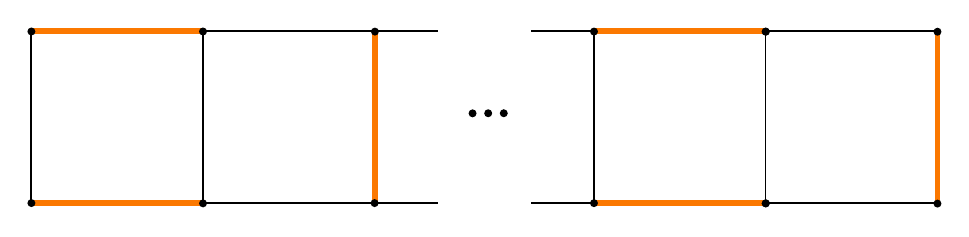}
    \caption{A perfect matching of the ladder graph $L_n$.}
    \label{fig:perfectmatching}
\end{figure}

In a perfect matching of a ladder graph, if one horizontal edge is included, then the horizontal edge directly above or below it must also be included. Notice that a perfect matching of a ladder graph is completely determined by which pairs of horizontal edges are contained in the perfect matching. Moreover, adjacent horizontal edges cannot be simultaneously included. Choosing a perfect matching of the ladder graph $L_r$ is therefore equivalent to choosing a subset of the integers $\left[1,r-1\right]$ without choosing any consecutive integers.

With this in mind we have the following combinatorial result, which is an important piece in a later proof.
\begin{theorem}[Musiker \& Propp, Theorem 3 of~\cite{MP06}]\label{thm:combin}
    The number of ways to choose a subset ${S\subset \{1,2,\dots, 2r-1\}}$ such that $S$ contains $\emph{a}$ odd elements, $\emph{b}$ even elements, and no consecutive elements is 
    \[ \binom{r-1-a}{b}\binom{r-b}{a}. \]
\end{theorem}
\begin{remark}
    This differs from the statement in~\cite{MP06} in the following ways: we require only the first of the two cases  (where Musiker and Propp's $N$ is odd), and we replace their $n$, $q$, and $r$ with $r-1$, $a$, and $b$, respectively.
\end{remark}

\section{Simplifying A-polynomial calculations}\label{sec:results} 

In this section we give precise statements of our results along with their proofs. First we consider the results that have analogues in the context of cluster algebras and later we see how this structure can be used to simplify the calculation of A-polynomials. 

\subsection{Results related to cluster algebras}  Recall that we use $n$ for the length of a \textit{tail} of a layered solid torus, which is the subset of tetrahedra corresponding to the maximal string of L's or R's preceding the tip of the word that describes its construction (see Definition~\ref{def:tail}). Also recall that the $k^{th}$ step is the step between triangles $T_k$ and $T_{k+1}$ in the Farey triangulation. Throughout this section, $k$ can be treated as fixed.

For ease of notation, define the following family of polynomials.

\begin{definition}\label{def:Hn}
    \[ H_n = \gamma_{f_k}^{2n} + \sum_{a+b\leq n-1} (-1)^{n-a-b}\binom{n-1-a}{b}\binom{n-b}{a}\gamma_{f_k}^{2a}\gamma_{o_k}^{2b}\gamma_{p_k}^{2(n-a-b)}, \text{ for } n\in\mathbb{Z}^+. \]
\end{definition}

\begin{theorem}\label{thm:laurent_precise}
    Suppose a layered solid torus has a tail of length $n\geq 1$ beginning at the $k^{th}$ step. Then, using the Ptolemy equations corresponding to each tetrahedron, the variable $\gamma_{h_{k+n-1}}$ can be expressed as 
    \[ \gamma_{h_{k+n-1}}=\frac{H_{n}}{\gamma_{f_k}^{n-1}\gamma_{o_k}^{n}}. \]
    Thus, $\gamma_{h_{k+n-1}}$ can be expressed as an integer Laurent polynomial in the variables $\gamma_{f_{k}}, \gamma_{o_{k}}$ and $\gamma_{p_{k}}$.
\end{theorem}

To prove this theorem we first establish a relationship between $H_n$ and the perfect matchings of a weighted ladder graph $G_r$. Let $G_r$ be the ladder graph $L_r$ with edges weighted as in Figure~\ref{fig:ladderwithweights}. Vertical edge weights alternate between $\gamma_{p_k}$ and $-\gamma_{p_k}$, starting with $\gamma_{p_k}$ on the left. Horizontal edge weights alternate between $\gamma_{f_k}$ and $\gamma_{o_k}$, starting with $\gamma_{f_k}$ on the left.

\begin{definition}
Let $\mathcal{S}$ be the set of all perfect matchings of the graph $G_r$. We define a polynomial $P_r$ in the variables $\gamma_{f_k}, \gamma_{o_k}$ and $\gamma_{p_k}$ to be the sum of the weights of all perfect matchings in $\mathcal{S}$. That is,
\[ P_{r}(\gamma_{f_k},\gamma_{o_k},\gamma_{p_k})= \sum_{S\in\mathcal{S}} w(S).\]
\end{definition}

\begin{figure}[ht]
    \centering
    \def\svgwidth{0.75\textwidth}
    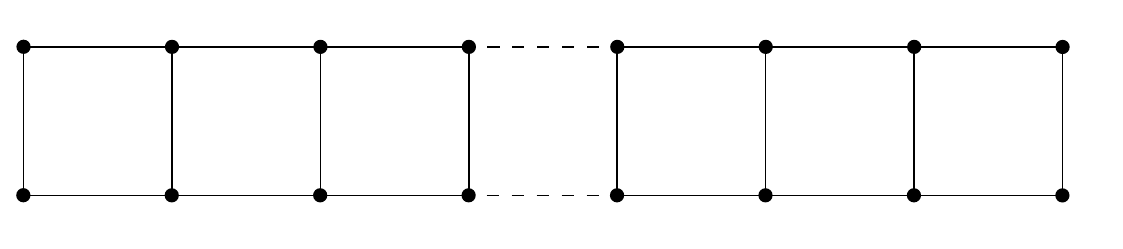
    \caption{The graph $G_r$ for even $r$, with edges weighted as described.}
    \label{fig:ladderwithweights}
\end{figure}

We show that $H_n$ is equivalent to $P_{2n}$.

\begin{lemma}[Musiker \& Propp, Lemma 2 of~\cite{MP06}]\label{lem:pms=oddeven}
    The number of ways to choose a perfect matching of $G_r$ with $\emph{a}$ pairs of edges weighted $\gamma_{f_k}$ and $\emph{b}$ pairs of edges weighted $\gamma_{o_k}$ is the number of ways to choose a subset $S\subset \{1,2,\dots, r-1\}$ such that $S$ contains $\emph{a}$ odd elements, $\emph{b}$ even elements, and no consecutive elements.
\end{lemma}
\begin{proof}
    To see this, note that all perfect matchings of $G_r$ can be found by choosing pairs of parallel horizontal edges with the condition that no consecutive edges are chosen. Pairs of parallel edges weighted $\gamma_{f_k}$ are in one-to-one correspondence with the odd integers between 1 and $r-1$, while pairs of parallel edges weighted $\gamma_{o_k}$ are in one-to-one correspondence with the even integers between 1 and $r-1$.
\end{proof}
\begin{remark}\label{rem:coefficient}
    Note that a perfect matching as described above must also include $\lceil r/2\rceil -a-b$ vertical edges each weighted $\gamma_{p_k}$ and $\lfloor r/2\rfloor -a-b$ vertical edges each weighted $-\gamma_{p_k}$. Hence, when $r=2n$, the number described in Lemma~\ref{lem:pms=oddeven} is the coefficient of the term $(-1)^{n-a-b}\gamma_{f_k}^{2a}\gamma_{o_k}^{2b}\gamma_{p_k}^{2(n-a-b)}$ in $P_{2n}$.
\end{remark}

Recall from Theorem~\ref{thm:combin} that the number of ways to choose a subset ${S\subset \{1,2,\dots, 2r-1\}}$ such that $S$ contains $a$ odd elements, $b$ even elements, and no consecutive elements is 
    \[ \binom{r-1-a}{b}\binom{r-b}{a}. \]

\begin{lemma}\label{lem:Hn=Pr}
For $H_n$ and $P_r$ as described above, we have $H_n=P_{2n}$.
\end{lemma}
\begin{proof}
    Consider the graph $G_{2n}$. With notation as above, observe that $a$ can range between $0$ and $n$, since there are $n$ odd integers between 1 and $2n-1$. Similarly, $b$ can range between $0$ and $n-1$, since there are $n-1$ even integers between 1 and $2n-1$. Moreover, since we cannot choose consecutive integers, the sum of $a$ and $b$ is at most $n-1$, except in the case where $b=0$ and $a=n$. With this, along with the observation in Remark~\ref{rem:coefficient}, we have 
    \[ P_{2n}= \gamma_{f_k}^{2n} + \sum_{a+b\leq n-1} (-1)^{n-a-b} \binom{n-1-a}{b}\binom{n-b}{a} \gamma_{f_k}^{2a}\gamma_{o_k}^{2b}\gamma_{p_k}^{2(n-a-b)}=H_n. \]
\end{proof}

Lemma~\ref{lem:pna} and Lemma~\ref{lem:pnb} establish recursive properties of the polynomials $P_r$.
\begin{lemma}\label{lem:pna}
    The polynomials $P_r$ satisfy the recurrence
    \[P_{2r}=P_{2r-2}(\gamma_{f_k}^2 + \gamma_{o_k}^2 - \gamma_{p_k}^2) - \gamma_{f_k}^2 \gamma_{o_k}^2 P_{2r-4}, \text{ for all } r\geq 3. \]
\end{lemma}
\begin{proof}
Assume $r\geq 3$. A perfect matching of $G_r$ can be considered as either: a perfect matching of $G_{r-1}$, plus the vertical edge at the far right of weight $-\gamma_{p_k}$ (if $r$ is even) or $\gamma_{p_k}$ (if $r$ is odd); or a perfect matching of $G_{r-2}$, plus the pair of horizontal edges on the far right, which are weighted either $\gamma_{f_k}$ (if $r$ is even) or $\gamma_{o_k}$ (if $r$ is odd).
This observation gives us the following: 
\begin{align*}
    P_{2r}&=-\gamma_{p_k} P_{2r-1} + \gamma_{f_k}^2 P_{2r-2}, \\
    P_{2r-1}&=\gamma_{p_k} P_{2r-2} + \gamma_{o_k}^2 P_{2r-3}, \\
    P_{2r-2}&=-\gamma_{p_k} P_{2r-3} + \gamma_{f_k}^2 P_{2r-4}.
\end{align*}

 We solve the first and third equations for $P_{2r-1}$ and $P_{2r-3}$, respectively, then substitute these into the second equation to get 
\begin{align*}
    \frac{\gamma_{f_k}^2 P_{2r-2}-P_{2r}}{\gamma_{p_k}}&=\gamma_{p_k} P_{2r-2} + \gamma_{o_k}^2  \frac{\gamma_{f_k}^2P_{2r-4}-P_{2r-2}}{\gamma_{p_k}} \nonumber \\
    \gamma_{f_k}^2 P_{2r-2} - P_{2r}&=\gamma_{p_k}^2 P_{2r-2} + \gamma_{o_k}^2 (\gamma_{f_k}^2 P_{2r-4} - P_{2r-2}) \nonumber \\
    P_{2r}&=\gamma_{f_k}^2 P_{2r-2} + \gamma_{o_k}^2 P_{2r-2} - \gamma_{p_k}^2 P_{2r-2} - \gamma_{f_k}^2 \gamma_{o_k}^2 P_{2r-4} \nonumber \\
    P_{2r}&=P_{2r-2}(\gamma_{f_k}^2 + \gamma_{o_k}^2 - \gamma_{p_k}^2) - \gamma_{f_k}^2 \gamma_{o_k}^2 P_{2r-4}. 
\end{align*}
\end{proof}

\begin{lemma}\label{lem:pnb}
The polynomials $P_r$ satisfy the recurrence
\[P_{2r-2}\cdot P_{2r-6}=P_{2r-4}^2-(\gamma_{f_k}^{r-3}\gamma_{o_k}^{r-2}\gamma_{p_{k}})^2, \text{ for } r\geq 4.\]
\end{lemma}
\begin{proof}
Note that 
\begin{equation*}
    P_2 = \gamma_{f_k}^{2}-\gamma_{p_k}^2 \qquad \text{ and } \qquad P_4 = \gamma_{f_k}^{4} + \gamma_{p_k}^{4} -2\gamma_{f_k}^{2}\gamma_{p_k}^{2} -\gamma_{o_k}^{2}\gamma_{p_k}^{2}.
\end{equation*}
All perfect matchings of $G_6$ are shown in Figure~\ref{fig:G6pms}. From this, we have that
   \[ P_6 = \gamma_{f_k}^{6} - \gamma_{o_k}^4\gamma_{p_k}^{2} - 2\gamma_{f_k}^{2}\gamma_{o_k}^{2}\gamma_{p_k}^{2} - 3\gamma_{f_k}^{4}\gamma_{p_k}^{2} + 2\gamma_{o_k}^{2}\gamma_{p_k}^{4} + 3\gamma_{f_k}^{2}\gamma_{p_k}^{4} - \gamma_{p_k}^{6}.\]

\begin{figure}
    \centering
    \includegraphics[width=0.8\textwidth]{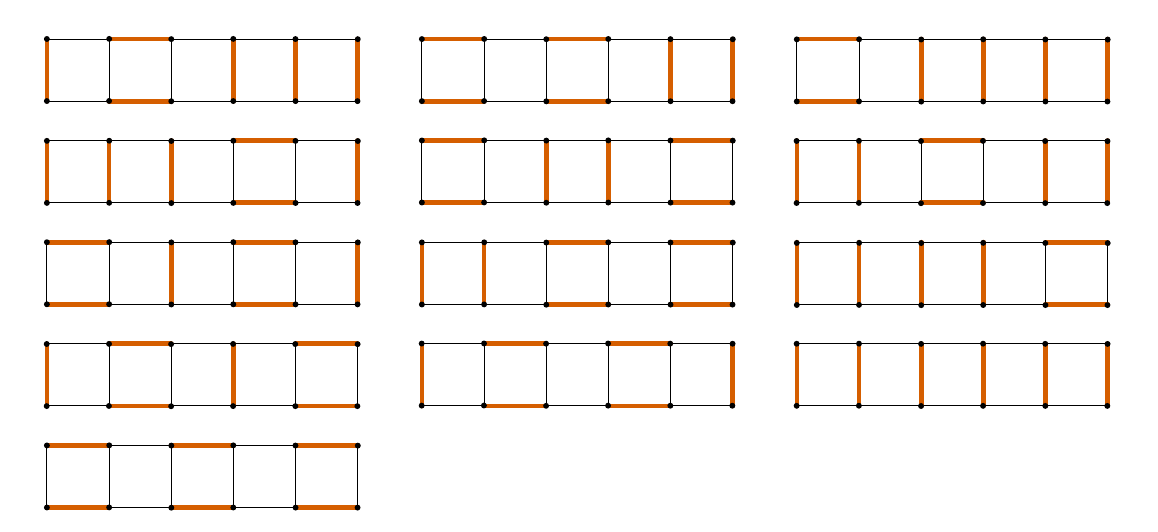}
    \caption{All perfect matchings of the graph $G_6$ (with weights omitted for clarity).}
    \label{fig:G6pms}
\end{figure}

So, when $r=4$, we have 
\begin{align*}
    P_6\cdot P_2&= (\gamma_{f_k}^{6} - \gamma_{o_k}^4\gamma_{p_k}^{2} - 2\gamma_{f_k}^{2}\gamma_{o_k}^{2}\gamma_{p_k}^{2} - 3\gamma_{f_k}^{4}\gamma_{p_k}^{2} + 2\gamma_{o_k}^{2}\gamma_{p_k}^{4} + 3\gamma_{f_k}^{2}\gamma_{p_k}^{4} - \gamma_{p_k}^{6}) \cdot (\gamma_{f_k}^{2}-\gamma_{p_k}^2) \\
    &= \gamma_{f_k}^{8} - \gamma_{f_k}^{2}\gamma_{o_k}^{4}\gamma_{p_k}^{2} - 2\gamma_{f_k}^{4}\gamma_{o_k}^{2}\gamma_{p_k}^{2} - 4\gamma_{f_k}^{6}\gamma_{p_k}^{2} + 4\gamma_{f_k}^{2}\gamma_{o_k}^{2}\gamma_{p_k}^{4} + 6\gamma_{f_k}^{4}\gamma_{p_k}^{4} - 4\gamma_{f_k}^{2}\gamma_{p_k}^{6} \\
    & \quad + \gamma_{o_k}^4\gamma_{p_k}^{4} - 2\gamma_{o_k}^{2}\gamma_{p_k}^{6} + \gamma_{p_k}^{8} \\
    &= (\gamma_{f_k}^{4} + \gamma_{p_k}^{4} -2\gamma_{f_k}^{2}\gamma_{p_k}^{2} -\gamma_{o_k}^{2}\gamma_{p_k}^{2})^2 - \gamma_{f_k}^{2}\gamma_{o_k}^{4}\gamma_{p_k}^{2} \\
    &= P_4^2 - \gamma_{f_k}^{2}\gamma_{o_k}^{4}\gamma_{p_k}^{2}.
\end{align*}
This establishes the base case. 

Now consider $r>4$ and assume for induction that 
\[ P_{2r-2}\cdot P_{2r-6}=P_{2r-4}^2-(\gamma_{f_k}^{r-3}\gamma_{o_k}^{r-2}\gamma_{p_{k}})^2.\]
Considering $P_{2(r+1)-2}\cdot P_{2(r+1)-6}$ we have 
\begin{align*}
    P_{2r}\cdot P_{2r-4} &= \left( P_{2r-2}(\gamma_{f_k}^2 + \gamma_{o_k}^2 - \gamma_{p_k}^2) - \gamma_{f_k}^2 \gamma_{o_k}^2 P_{2r-4} \right) \cdot P_{2r-4}, & \text{ from Lemma~\ref{lem:pna}} \\
    &= P_{2r-2}\cdot P_{2r-4}(\gamma_{f_k}^2 + \gamma_{o_k}^2 - \gamma_{p_k}^2) - \gamma_{f_k}^2 \gamma_{o_k}^2 P_{2r-4}^2  &  \\
    &= P_{2r-2}\cdot P_{2r-4}(\gamma_{f_k}^2 + \gamma_{o_k}^2 - \gamma_{p_k}^2) \\
    &\qquad - \gamma_{f_k}^2 \gamma_{o_k}^2 \left(P_{2r-2}\cdot P_{2r-6}+(\gamma_{f_k}^{r-3}\gamma_{o_k}^{r-2}\gamma_{p_{k}})^2\right) & \text{by assumption} \\
    &= P_{2r-2}\cdot \left( P_{2r-4}(\gamma_{f_k}^2 + \gamma_{o_k}^2 - \gamma_{p_k}^2) - \gamma_{f_k}^2 \gamma_{o_k}^2 P_{2r-6}\right) \\
    &\qquad - \gamma_{f_k}^2 \gamma_{o_k}^2(\gamma_{f_k}^{r-3}\gamma_{o_k}^{r-2}\gamma_{p_{k}})^2 & \\
    &= P_{2r-2}\cdot P_{2r-2} - (\gamma_{f_k}^{r-2}\gamma_{o_k}^{r-1}\gamma_{p_{k}})^2 & \text{ from Lemma~\ref{lem:pna}.}
\end{align*}
Hence, by induction, 
\[P_{2r-2}\cdot P_{2r-6}=P_{2r-4}^2-(\gamma_{f_k}^{r-3}\gamma_{o_k}^{r-2}\gamma_{p_{k}})^2, \text{ for } r\geq 4.\]
\end{proof}

\begin{lemma}\label{lem:Hrecurrence}
    The polynomials $H_n$ satisfy the following recurrence, for any $n\geq 3$.
    \[ H_{n}\cdot H_{n-2} = H_{n-1}^2-(\gamma_{f_k}^{n-2}\gamma_{o_k}^{n-1}\gamma_{p_{k}})^2. \]
\end{lemma}

\begin{proof}
 We have $H_n=P_{2n}$ by Lemma~\ref{lem:Hn=Pr}, so the result follows by setting $r=n+1$ in Lemma~\ref{lem:pnb}.
\end{proof}

We are now in a position to prove Theorem~\ref{thm:laurent_precise}.  Recall that the equations involved in this proof are those associated with the tail of the layered solid torus, and we ignore equations related to the head, body and tip of the layered solid torus (recall Definition~\ref{def:tail}).

\begin{proof}[Proof of Theorem~\ref{thm:laurent_precise}]
    We proceed by induction on $n$, the length of the tail. 
    
    When $n=1$, the tail of the layered solid torus consists of one tetrahedron. The corresponding Ptolemy equation (from Theorem~\ref{thm:HMP}) is  the one for the $k^{th}$ step: $\gamma_{o_k}\gamma_{h_k}+\gamma_{p_k}^2-\gamma_{f_k}^2=0$, which we rewrite as
    \begin{equation}\label{eq:hk}
        \gamma_{h_{k}}=\frac{\gamma_{f_{k}}^2-\gamma_{p_{k}}^2}{\gamma_{o_{k}}}.
    \end{equation}
     Recalling Definition~\ref{def:Hn}, we have 
    \[ H_1= \gamma_{f_k}^{2} + \sum_{a=b=0} (-1)^{1-a-b}\binom{0-a}{b}\binom{1-b}{a}\gamma_{f_k}^{2a}\gamma_{o_k}^{2b}\gamma_{p_k}^{2(1-a-b)}= \gamma_{f_k}^{2}-\gamma_{p_k}^2, \]
    so we have 
    \[\gamma_{h_{k}}=\frac{H_1}{\gamma_{o_{k}}}.\]
    
    When $n=2$, the tail of the layered solid torus consists of two tetrahedra and the corresponding Ptolemy equations are  those corresponding to the $k^{th}$ and $(k+1)^{st}$ steps, namely ${\gamma_{o_k}\gamma_{h_k}+\gamma_{p_k}^2-\gamma_{f_k}^2=0}$ (as above) and ${\gamma_{o_{k+1}}\gamma_{h_{k+1}}+\gamma_{p_{k+1}}^2-\gamma_{f_{k+1}}^2=0}$. Rearranging the second equation gives
    \[ \gamma_{h_{k+1}}=\frac{\gamma_{f_{k+1}}^2-\gamma_{p_{k+1}}^2}{\gamma_{o_{k+1}}}. \]
     However, in a tail we know that certain slopes are equal, as seen in Figure~\ref{fig:eq_slopes}.
    
    \begin{figure}
        \centering
        \includegraphics[width=0.75\textwidth]{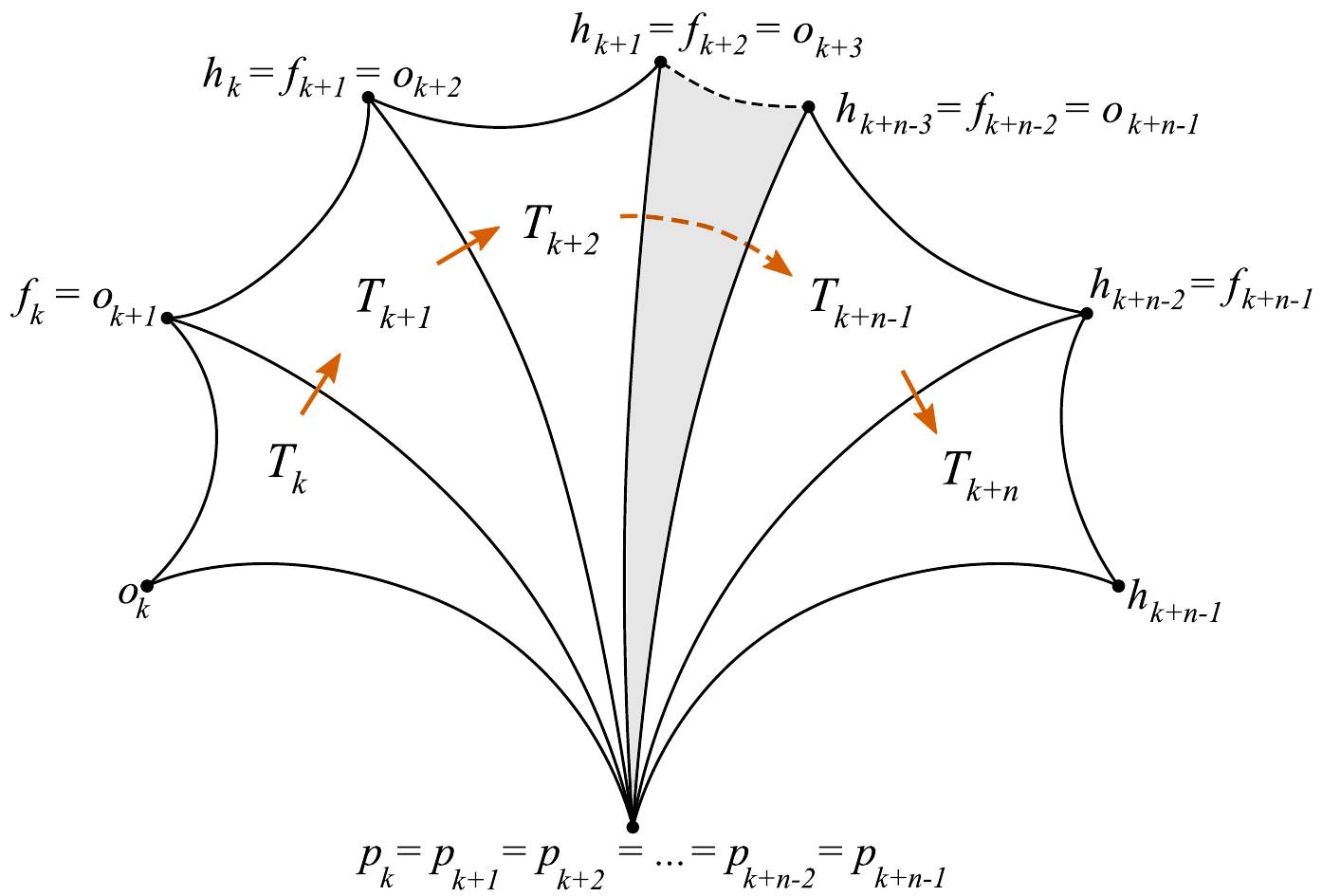}
        \caption{A tail of length $n$ beginning at step $k$. The tail starts in triangle $T_k$ and ends in triangle $T_{k+n}$. Each vertex is labelled by multiple slopes, since they are relevant to multiple steps in the tail (except for the vertices corresponding to $o_k$ and $h_{k+n-1}$). In particular, all pivot slopes are the same, the heading slope for one step is the fan slope for the next, and the fan slope for one step is the old slope for the next.}
        \label{fig:eq_slopes}
    \end{figure}
    
    In particular, we have
    \begin{equation*}
        \gamma_{p_{k+1}}=\gamma_{p_k}, \qquad \gamma_{f_{k+1}}=\gamma_{h_{k}}, \qquad \text{and}\qquad \gamma_{o_{k+1}}=\gamma_{f_{k}}.
    \end{equation*}
    Hence, making these substitutions and using equation \eqref{eq:hk}, we have 
    \begin{equation*}
        \gamma_{h_{k+1}}=\frac{\gamma_{h_{k}}^2-\gamma_{p_{k}}^2}{\gamma_{f_{k}}}=\frac{\Big(\frac{\gamma_{f_{k}}^2-\gamma_{p_{k}}^2}{\gamma_{o_{k}}}\Big)^2-\gamma_{p_{k}}^2}{\gamma_{f_{k}}}= \frac{\gamma_{f_k}^{4} + \gamma_{p_k}^{4} -2\gamma_{f_k}^{2}\gamma_{p_k}^{2} -\gamma_{o_k}^{2}\gamma_{p_k}^{2}}{\gamma_{f_{k}}\gamma_{o_k}^2}. 
    \end{equation*}

    Meanwhile, 
    \begin{align*}
        H_2 &= \gamma_{f_k}^{4} + \sum_{a+b\leq 1} (-1)^{2-a-b}\binom{1-a}{b}\binom{2-b}{a}\gamma_{f_k}^{2a}\gamma_{o_k}^{2b}\gamma_{p_k}^{2(2-a-b)} \\
        &= \gamma_{f_k}^{4} + \gamma_{p_k}^{4} -2\gamma_{f_k}^{2}\gamma_{p_k}^{2} -\gamma_{o_k}^{2}\gamma_{p_k}^{2}.
    \end{align*}
    Thus, 
    \[ \gamma_{h_{k+1}} = \frac{H_2}{\gamma_{f_{k}}\gamma_{o_k}^2}. \]
    
    Now, suppose $n>2$ and assume for induction that  
    \[ \gamma_{h_{k+i-1}}=\frac{H_{i}}{\gamma_{f_k}^{i-1}\gamma_{o_k}^{i}}, \text{ for all } i<n. \]
    In a tail of length $n$ there are $n$ tetrahedra. The Ptolemy equation corresponding to the $n^{th}$ tetrahedron is  the one from Theorem~\ref{thm:HMP} associated to the $(k+n-1)^{st}$ step, which can be written as 
    \begin{equation}
        \gamma_{h_{k+n-1}}=\frac{\gamma_{f_{k+n-1}}^2-\gamma_{p_{k+n-1}}^2}{\gamma_{o_{k+n-1}}} \label{hkm-2A}.
    \end{equation}
     Again, with reference to Figure~\ref{fig:eq_slopes}, observe that the following variables are equivalent in the tail:
    \[\gamma_{p_{k+n-1}}=\gamma_{p_{k}}, \quad  \gamma_{f_{k+n-1}}=\gamma_{h_{k+n-2}}, \quad \text{and} \quad \gamma_{o_{k+n-1}}=\gamma_{h_{k+n-3}},\] 
    so \eqref{hkm-2A} becomes 
    \[ \gamma_{h_{k+n-1}}=\frac{\gamma_{h_{k+n-2}}^2-\gamma_{p_{k}}^2}{\gamma_{h_{k+n-3}}}. \]
    
    Now, using the inductive assumption we write
    \begin{align*}
    \gamma_{h_{k+n-1}}&=\Bigg[\Bigg(\frac{{H_{n-1}}}{\gamma_{f_k}^{n-2}\gamma_{o_k}^{n-1}} \Bigg)^2-\gamma_{p_{k}}^2 \Bigg]{ \Bigg/ \Bigg(  \frac{H_{n-2}}{\gamma_{f_k}^{n-3}\gamma_{o_k}^{n-2}}} \Bigg)\\
    &=\frac{H_{n-1}^2-(\gamma_{f_k}^{n-2}\gamma_{o_k}^{n-1}\gamma_{p_{k}})^2}{\gamma_{f_k}^{n-1}\gamma_{o_k}^{n} H_{n-2}}.   
    \end{align*}
    
    Hence, to prove the result, we need 
    \[ \frac{H_{n-1}^2-(\gamma_{f_k}^{n-2}\gamma_{o_k}^{n-1}\gamma_{p_{k}})^2}{\gamma_{f_k}^{n-1}\gamma_{o_k}^{n} H_{n-2}} = \frac{H_{n}}{\gamma_{f_k}^{n-1}\gamma_{o_k}^{n}}. \]
    But this is equivalent to showing that 
    \[H_{n}\cdot H_{n-2} = H_{n-1}^2-(\gamma_{f_k}^{n-2}\gamma_{o_k}^{n-1}\gamma_{p_{k}})^2,\]
    for $n>2$, which is the recurrence in Lemma~\ref{lem:Hrecurrence}.
    Hence, the claim follows by induction. 
\end{proof}

\begin{theorem}\label{thm:intnos_precise}
    Suppose the tail of a layered solid torus has length $n\geq 1$ and begins at step $k$. Let $\alpha_s$ be the geodesic in $\mathbb{H}^2$ whose endpoints are the vertices corresponding to the slopes $h_{k+n-1}$ (the heading slope at the end of the tail) and $s$, where $s$ is one of $f_k, o_k$ or $p_k$ (the fan, old and pivot slopes at the beginning of the tail). The exponent of $\gamma_{s}$ in the denominator of the Laurent polynomial for $\gamma_{h_{k+n-1}}$ is given by the intersection number of $\alpha_s$ with edges it intersects in the Farey triangulation. 
\end{theorem}
\begin{proof}
    Denote the set of edges in the Farey triangulation by $\mathcal{F}$ and let $|\alpha_s\cap\mathcal{F}|$ be the number of transverse intersections between the geodesic $\alpha_s$ and all edges in $\mathcal{F}$. In each of the accompanying figures, $\alpha_{f_k}$ is shown in dark blue, $\alpha_{o_k}$ is shown in green, and $\alpha_{p_k}$ is shown in light blue. 
    
    For this proof we consider the Farey triangulation of the upper half-space model of $\mathbb{H}^2$. After applying the appropriate (not necessarily orientation-preserving) isometry of $\mathbb{H}^2$, we may assume that $f_k=1/0$, $o_k=-1/1$, $p_k=0/1$ and $h_k=1/1$. Note that this choice of slopes ensures that $h_{k+n-1}=1/n$ for all $n\geq 1$. In other words, when considering a tail of length $n$, the common endpoint of $\alpha_{f_k},\alpha_{o_k}$ and $\alpha_{p_k}$ is $1/n$. 

    We prove the claim by induction on the length of the tail. 
    When $n=1$ we have the situation shown in Figure~\ref{fig:int_basecase}. In particular, we see that $\alpha_{f_k}$ and $\alpha_{p_k}$ are each parallel to edges in the Farey triangulation and therefore $|\alpha_{f_k}\cap\mathcal{F}|=|\alpha_{p_k}\cap\mathcal{F}|=0$. Meanwhile, $\alpha_{o_k}$ intersects one edge in the Farey triangulation so $|\alpha_{o_k}\cap\mathcal{F}|=1$. From Theorem~\ref{thm:laurent_precise}, we know that the denominator of the Laurent polynomial for $\gamma_{h_{k}}$ is 
    \[\gamma_{f_k}^{0}\gamma_{o_k}^{1}\gamma_{p_k}^{0}=\gamma_{f_k}^{|\alpha_{f_k}\cap\mathcal{F}|}\gamma_{o_k}^{|\alpha_{o_k}\cap\mathcal{F}|}\gamma_{p_k}^{|\alpha_{p_k}\cap\mathcal{F}|}. \]
    Hence, the base case holds.
    
    \begin{figure}[ht]
        \centering
        \includegraphics[width=0.45\textwidth]{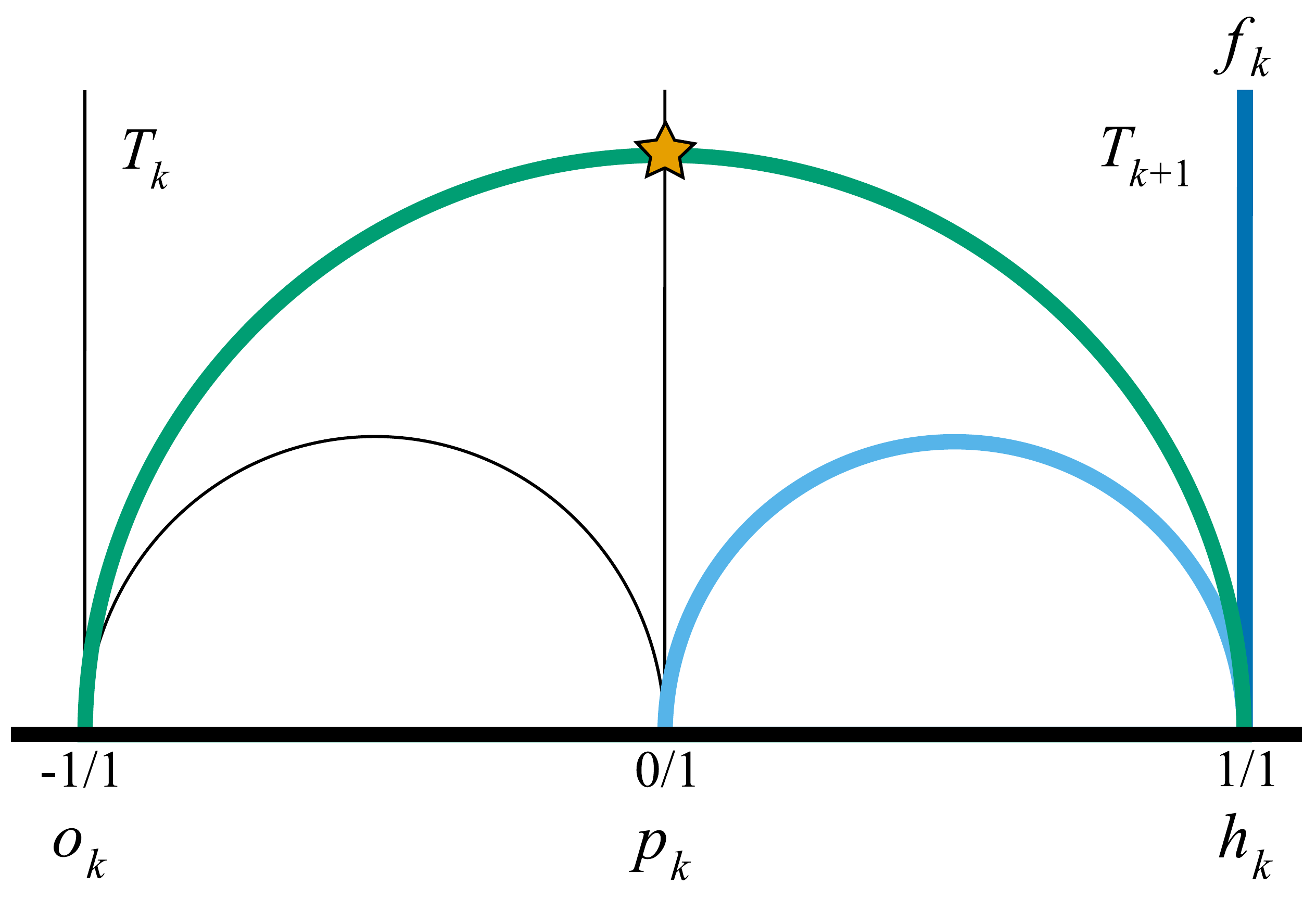}
        \caption{The geodesics corresponding to a tail of length $n=1$} beginning at step $k$.  
        The star indicates the intersection between $\alpha_{o_k}$ and $\mathcal{F}$.
        \label{fig:int_basecase}
    \end{figure}

    In Figure~\ref{fig:int_induction} we see that increasing the length of the tail by 1 increases each of $|\alpha_{o_k}\cap\mathcal{F}|$ and $|\alpha_{f_k}\cap\mathcal{F}|$ by 1, while $|\alpha_{p_k}\cap\mathcal{F}|$ is always 0. Hence, the claim follows by induction.
    
    \begin{figure}[ht]
        \centering
        \includegraphics[width=\textwidth]{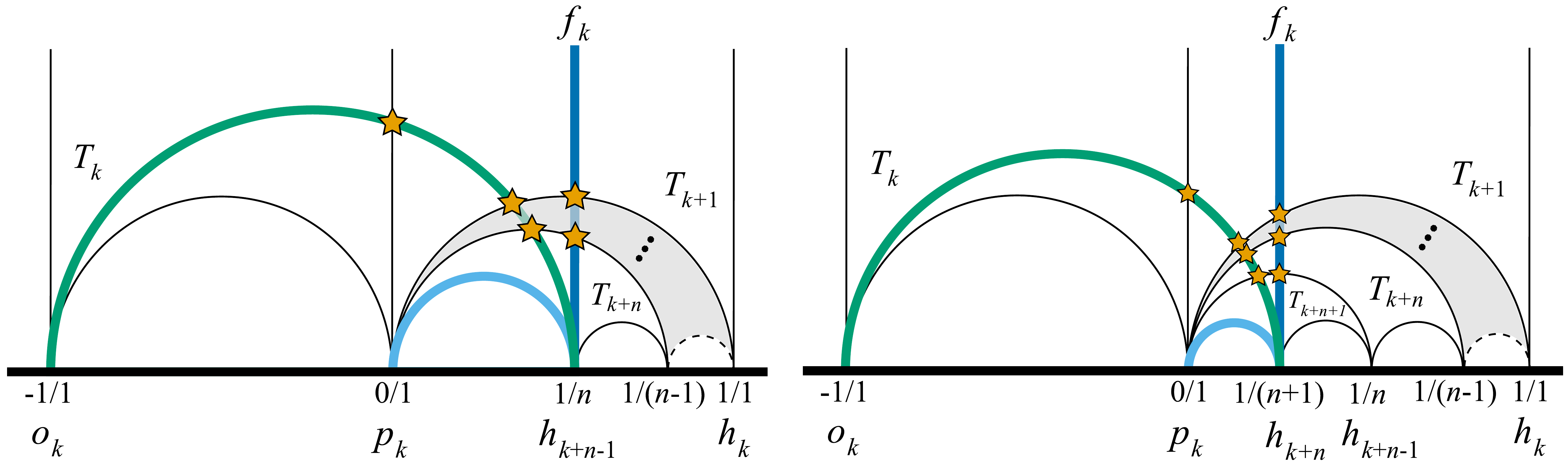}
        \caption{
        The geodesics corresponding to a tail of length $n$ (left) and length $n+1$ (right) beginning at step $k$. The shaded segments are where one should imagine the fan of triangles $T_{k+2}$ through $T_{k+n-1}$. Stars indicate transverse intersections. Note that ${|\alpha_{p_k}\cap\mathcal{F}|=0}$ for both the tails of length $n$ and $n+1$. Meanwhile, each of ${|\alpha_{o_k}\cap\mathcal{F}|}$ and ${|\alpha_{f_k}\cap\mathcal{F}|}$ increase by 1 as the length of the tail increases by 1.}
        \label{fig:int_induction}
    \end{figure}
\end{proof}

\subsection{Applications to A-polynomial calculations}\label{sec:Apolyresult}

 Recall that, apart from changing the labels of slopes from absolute to relative directions, our equations are the same as those in~\cite{HMP}. As such, from Theorem 2.58 in~\cite{HMP}, we know that setting one of the $\gamma$ variables to 1 and solving the system of Ptolemy equations gives the geometric factor of the PSL$(2,\mathbb{C})$ A-polynomial. Moreover, by first making appropriate substitutions for the variables corresponding to the meridian and longitude, we obtain a rational function in $L$ and $M$ that contains the geometric factor of the SL$(2,\mathbb{C})$ A-polynomial (see Corollary 2.59 in~\cite{HMP}). However, finding solutions to such a system directly is again impeded by the increasing number of equations as the triangulation grows. Fortunately, we may use Theorem~\ref{thm:laurent_precise} to simplify this computation.

\begin{theorem}\label{thm:Apoly_precise}
Suppose a knot is obtained from a link complement by Dehn filling using a layered solid torus. Suppose the tail of the layered solid torus has length $n\geq 1$ and begins at step $k$, and suppose the folding equation corresponds to the tip being in the same direction as the tail. The folding equation, along with the set of tail equations, is equivalent to the equation
\[ H_{n} - \gamma_{f_k}^{n-1}\gamma_{o_k}^{n}\gamma_{p_k} = 0. \]
\end{theorem} 

\begin{proof}
    In the proof of Theorem~\ref{thm:laurent_precise}, we saw that the set of $n$ tail equations is equivalent to the equation
    \[ \gamma_{h_{k+n-1}}=\frac{H_{n}}{\gamma_{f_k}^{n-1}\gamma_{o_k}^{n}}. \]
    The folding equation corresponding to the tip in the same direction as the tail is $\gamma_{p_{k+n}}=\gamma_{f_{k+n}}$  (from Theorem~\ref{thm:HMP} with $N=n+k$). However, recall from Figure~\ref{fig:eq_slopes} that: all pivot slopes in the tail are equal, so $\gamma_{p_{k+n}}=\gamma_{p_{k}}$; and the fan slope of the $(k+n)^{th}$ step is equal to the heading slope of the $(k+n-1)^{st}$ step, so $\gamma_{f_{k+n}}=\gamma_{h_{k+n-1}}$. Hence, we set $\gamma_{h_{k+n-1}}$ equal to $\gamma_{p_k}$ and rearrange to get
    \[ H_{n} - \gamma_{f_k}^{n-1}\gamma_{o_k}^{n}\gamma_{p_k} = 0, \text{ as desired.} \]
\end{proof}

 The above result consolidates all equations associated with the tip and tail of a layered solid torus, however, to compute A-polynomials, we also require: the finitely many equations coming from the head and body of the layered solid torus; and the finitely many equations corresponding to the tetrahedra that triangulate the parent link. As seen in~\cite{HMP}, when the meridian or longitude intersect a tetrahedron in the parent link, the corresponding equation involves the variables $L$ and $M$. Such equations can be used to express $\gamma$ variables in terms of $L$ and $M$.

\begin{corollary}\label{cor:Apoly_precise}
    When $\gamma_{f_k},\gamma_{o_k}$, and $\gamma_{p_k}$ are expressed in terms of $L$ and $M$ (using the equations from the parent link and the body of the Dehn filling), the rational function ${H_{n} - \gamma_{f_k}^{n-1}\gamma_{o_k}^{n}\gamma_{p_k}}$ contains the geometric factor of the SL$(2,\mathbb{C})$ A-polynomial for the corresponding knot. 
\end{corollary}
\begin{proof}
    This follows from the previous theorem and Corollary 2.59 of~\cite{HMP}.
\end{proof}

Once we have $\gamma_{f_k},\gamma_{o_k}$, and $\gamma_{p_k}$ expressed in terms of $L$ and $M$, Theorem~\ref{thm:Apoly_precise} gives a family of rational functions in $L$ and $M$ depending only on $n$. This means that the main barrier to effective computation of these rational functions is in finding $\gamma_{f_k},\gamma_{o_k}$, and $\gamma_{p_k}$ in terms of $L$ and $M$, which depends on the Ptolemy equations of the tetrahedra required to triangulate the parent link complement. In particular, this means that if a parent link admits an appropriate triangulation consisting of few tetrahedra, the A-polynomials for fillings of this link are readily computable. We now demonstrate the power of this result by applying it in the context of two families of knots related by twisting.

\section{Example calculations}\label{sec:examples}
In this section we see how Theorem~\ref{thm:Apoly_precise} can be applied to A-polynomial calculations for two families of knots related by twisting: the twisted torus knots $T(5,1-5m,2,2)$ and the twist knots $J(2,2m)$.  Throughout this section the variable $m$ is used in relation to $1/m$ Dehn fillings and the variable $n$ is used with reference to the length of a tail in a layered solid torus.

\subsection{A family of twisted torus knots}\label{sec:WHsis}

The $(-2,3,8)$-pretzel link, shown in Figure~\ref{fig:WHsis} (left), is a two-component link with a simple triangulation. It may also be presented as an augmented twisted torus knot as in Figure~\ref{fig:WHsis} (right). 
Notice that one of the components (shown in blue) is unknotted. By performing $1/m$ Dehn fillings on the blue component we generate the infinite family of twisted torus knots $T(5,1-5m,2,2)$~\cite{HMPT}. Note that we are using the twisted torus knot notation used in~\cite{ckm}.

\begin{figure}
    \centering
    \includegraphics[width=0.6\textwidth]{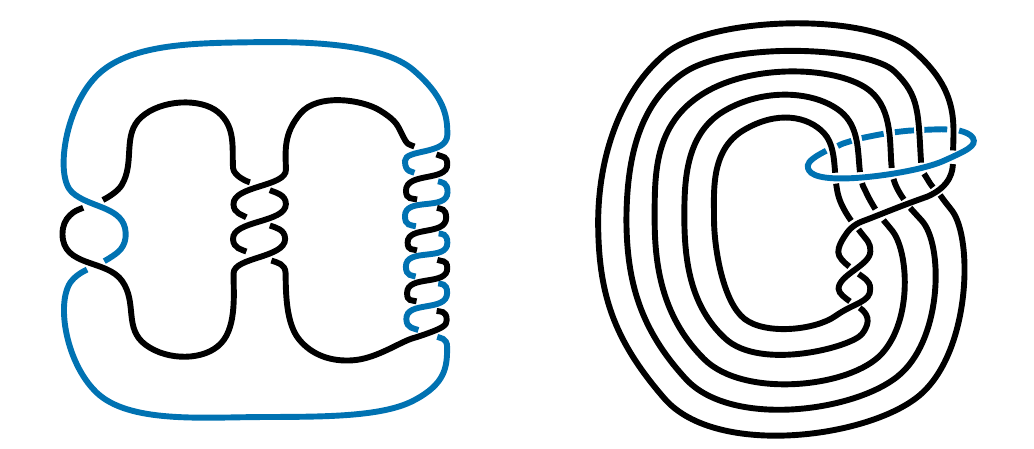}
    \caption{The $(-2,3,8)$-pretzel link with unknotted component shown in blue. On the left is the standard pretzel diagram and on the right, we see an alternative diagram that elucidates why $1/m$ Dehn fillings generate the $T(5,1-5m,2,2)$ twisted torus knots.}
    \label{fig:WHsis}
\end{figure}

Howie, Mathews, Purcell and the author study this link in~\cite{HMPT}, using the triangulation given in Figure~\ref{fig:WHsisregtri} (with notation as in Regina~\cite{Regina}). Observe that tetrahedra 2 and 3 glue only to each other and one face each of tetrahedra 0 and 1. Vertices 2(0) and 3(0) meet the unknotted cusp and all other vertices meet the other cusp. To perform Dehn fillings on the unknotted cusp we remove tetrahedra 2 and 3, leaving a once-punctured torus boundary triangulated by the faces 0(012) and 1(012). We may then glue an appropriate layered solid torus to these exposed faces to make the Dehn filling slope homotopically trivial.

\begin{figure}
\begin{center}
\begin{tabular}{c|c|c|c|c}
    Tetrahedron & Face 012 & Face 013 & Face 023 & Face 123 \\
    \hline
    0 & 2(312) & 1(023) & 1(312) & 1(031) \\
    1 & 3(123) & 0(132) & 0(013) & 0(230) \\
    2 & 3(021) & 3(031) & 3(032) & 0(120) \\
    3 & 2(021) & 2(031) & 2(032) & 1(012) 
\end{tabular}
\end{center}
\caption{A triangulation of the $(-2,3,8)$-pretzel link complement in Regina notation.}
\label{fig:WHsisregtri}
\end{figure}

 The Ptolemy equations for the outside tetrahedra were determined in~\cite{HMPT} to be 
\begin{align}
    M\gamma_{1/0}\gamma_{4/1} - M^2\gamma_{4/1}\gamma_{3/1} - L\gamma_{3/1}^2 &=0 \text{ for tetrahedron 0, and } \label{eq:tet0}\\
    - M^2\gamma_{3/1}^2 + L M\gamma_{1/0}\gamma_{4/1} - L \gamma_{3/1}\gamma_{4/1} &=0 \text{ for tetrahedron 1}. \label{eq:tet1}
\end{align}
These $\gamma$ variables are labelled according to the slopes of the corresponding edge classes; determining these slopes (namely, $3/1$, $4/1$, and $1/0$) is a non-trivial task that was done in~\cite{HMPT}. These equations differ slightly from the equations of~\cite{HMPT}, since here we have multiplied through by powers of $L$ and $M$ to remove negative exponents.

\subsubsection{Using the Farey triangulation}
To apply Theorem~\ref{thm:Apoly_precise}, we must determine paths in the Farey triangulation describing the construction of appropriate layered solid tori. This was done in~\cite{HMPT}.  
 Since the slopes of the boundary edges are $3/1$, $4/1$, and $1/0$, the starting Farey triangle is the one with vertices labelled by these rational numbers.

To perform $+1/m$ Dehn fillings we follow the path indicated in blue in Figure~\ref{fig:Fareypaths} and to perform $-1/m$ Dehn fillings we follow the path indicated in orange.  These paths can be described by the words L$^{2}$RL$^{m-2}$ and L$^{3}$R$^{m-1}$, respectively. Recall that the final L or R corresponds to the tip representing the fold, so this means that the layered solid torus used for a $+1/m$ Dehn filling has a tail of length $n=m-3$, while the layered solid torus used for a $-1/m$ Dehn filling has a tail of length $n=m-2$.

\begin{figure}
    \centering
    \includegraphics[width=0.6\textwidth]{Figures/WHsisWalks.pdf}
    \caption{Paths in the Farey diagram corresponding to $\pm 1/m$ Dehn fillings of the $(-2,3,8)$-pretzel link. The starting triangle is shaded grey and the triangles where each of the tails begin are indicated in blue and orange, for the positive and negative Dehn fillings, respectively. The twisted torus knot obtained by performing each Dehn filling is also noted.}
    \label{fig:Fareypaths}
\end{figure} 

Theorem~\ref{thm:Apoly_precise} applies for tails of length $n\geq 1$ with the tip in the same direction, so here we consider Dehn fillings with slopes $+1/m$ for $m\geq 4$ and $-1/m$ for $m\geq 3$. The tails for both the positive and negative Dehn fillings each start at step~4. Steps 0, 1 and 2 are the same for both paths, and using Theorem~\ref{thm:HMP}, we obtain their corresponding Ptolemy equations 
\begin{align}
    \gamma_{4/1}\gamma_{2/1}+\gamma_{3/1}^2-\gamma_{1/0}^2&=0, \label{eq:step0}\\
    \gamma_{3/1}\gamma_{1/1}+\gamma_{1/0}^2-\gamma_{2/1}^2&=0, \label{eq:step1}\\
    \gamma_{2/1}\gamma_{0/1}+\gamma_{1/0}^2-\gamma_{1/1}^2&=0. \label{eq:step2}
\end{align}

For positive Dehn fillings, step 3 is a right step and hence corresponds to the Ptolemy equation
\begin{equation}
    \gamma_{1/0}\gamma_{1/2}+\gamma_{1/1}^2-\gamma_{0/1}^2=0. \label{eq:step3pos}
\end{equation}
The tail begins at step 4 and we have $o_4=1/1$, $p_4=0/1$ and $f_4=1/2$. Hence, by Theorem~\ref{thm:Apoly_precise}, the equations for the tail of length $n\geq 1$  (corresponding to Dehn fillings of slope $+1/m$ for $m\geq 4$) are equivalent to the equation
\begin{equation}
    \gamma_{1/2}^{2n} + \sum_{a+b\leq n-1} (-1)^{n-a-b}\binom{n-1-a}{b}\binom{n-b}{a}\gamma_{1/2}^{2a}\gamma_{1/1}^{2b}\gamma_{0/1}^{2(n-a-b)} - \gamma_{1/1}^{n}\gamma_{1/2}^{n-1}\gamma_{0/1} = 0. \label{eq:tailpos}
\end{equation}

For negative Dehn fillings, step 3 is a left step and therefore corresponds to the Ptolemy equation
\begin{equation}
    \gamma_{1/1}\gamma_{-1/1}+\gamma_{1/0}^2-\gamma_{0/1}^2=0. \label{eq:step3neg}
\end{equation}
The tail begins at step 4 and we have $o_4=1/0$, $p_4=0/1$ and ${f_4=-1/1}$. Hence, by Theorem~\ref{thm:Apoly_precise}, the equations for the tail of length $n\geq 1$  (corresponding to Dehn fillings of slope $-1/m$ for $m\geq 3$) are equivalent to the equation
\begin{equation}
    \gamma_{-1/1}^{2n} + \sum_{a+b\leq n-1} (-1)^{n-a-b}\binom{n-1-a}{b}\binom{n-b}{a}\gamma_{-1/1}^{2a}\gamma_{1/0}^{2b}\gamma_{0/1}^{2(n-a-b)} - \gamma_{1/0}^{n}\gamma_{-1/1}^{n-1}\gamma_{0/1} = 0. \label{eq:tailneg}
\end{equation}

\subsubsection{A-polynomials for positive Dehn fillings}
The equations \eqref{eq:tet0} through \eqref{eq:tailpos} define a rational function that contains the geometric factor of the A-polynomial of the knot obtained by  $1/(n+3)$ Dehn filling of the $(-2,3,8)$-pretzel link, for $n\geq 1$.

We set $\gamma_{3/1}=1$ and use equations~\eqref{eq:tet0} through \eqref{eq:step3pos} to write $\gamma_{0/1},\gamma_{1/1}$ and $\gamma_{1/2}$ entirely in terms of $L$ and $M$. These can be found in Appendix~\ref{app:gammas}. With these substitutions, \eqref{eq:tailpos} becomes a formula for rational functions that contain the geometric factor of the A-polynomials for the twisted torus knots  $T(5,1-5(n+3),2,2)$, with $n\geq 1$. 

\subsubsection{A-polynomials for negative Dehn fillings}
The equations~\eqref{eq:tet0} through~\eqref{eq:step2}, along with equations~\eqref{eq:step3neg} and~\eqref{eq:tailneg} define a rational function that contains the geometric factor of the A-polynomial of the knot obtained by  $-1/(n+2)$ Dehn filling of the $(-2,3,8)$-pretzel link, for $n\geq 1$.

Again, we set $\gamma_{3/1}=1$ and use equations~\eqref{eq:tet0} through~\eqref{eq:step2} and equation~\eqref{eq:step3neg} to write $\gamma_{1/0},\gamma_{-1/1}$ and $\gamma_{0/1}$ entirely in terms of $L$ and $M$. These can also be found in Appendix~\ref{app:gammas}. With these substitutions, \eqref{eq:tailneg} becomes a formula for rational functions that contain the geometric factor of the A-polynomials for the twisted torus knots  $T(5,1+5(n+2),2,2)$, with $n\geq 1$. 

\subsubsection{Changing basis}
As discussed in~\cite{HMPT}, the choice of generators for the cusp homology were not the standard basis for the link in $S^3$. While we used the actual meridian, we did not use the preferred longitude. For the positive Dehn fillings, the required change of basis in the A-polynomial variables is  $(L,M)\mapsto (LM^{8-25m},M)$ for each $m=n+3\geq 4$. For the negative Dehn fillings, the required change of basis is  $(L,M)\mapsto (LM^{8+25m},M)$ for each $m=n+2\geq 3$.

\subsubsection{Comparing with what is known}
The twisted torus knots $T(5,16,2,2)$ and $T(5,-19,2,2)$ are equivalent to the census knots $K7_3$ and $K7_4$, respectively. After changing basis as above, and multiplying through by powers of $L$ and $M$ to remove negative exponents, the largest factors seen in the output of our formulas match the A-polynomials found by Culler~\cite{Culler}. For example, with substitutions as given in Appendix~\ref{app:gammas}, and with $n=1$, equation~\eqref{eq:tailneg} gives a polynomial with four factors, the largest of which has 455 terms. After changing basis, this factor is identical to the A-polynomial given for $K7_3$ on Culler's website~\cite{Culler}.

The knots $T(5,21,2,2)$ and $T(5,-24,2,2)$ are equivalent to the census knots $K8_3$ and $K8_4$, respectively, and our formula immediately gives a rational function containing the geometric factor of their A-polynomials despite their very large size; the largest factors of these have 784 and 952 terms, respectively. Since the A-polynomials for these knots do not appear on Culler's database we cannot compare.

\subsection{The twist knots}\label{sec:WHlink}

It is well known that the twist knots may be obtained by $1/m$ Dehn fillings of the complement of the Whitehead link (shown in Figure~\ref{fig:WHlink}, left). We use the notation $J(2,l)$ to mean the twist knot with $l$ right-handed crossings in the bottom twist region (as seen in Figure~\ref{fig:WHlink}, right). The $+1/m$ Dehn fillings of the Whitehead link therefore generate the family of twist knots $J(2,2m)$ and $-1/m$ Dehn fillings generate the family $J(2,-2m)$. 

\begin{figure}
    \centering
    \includegraphics[width=0.6\textwidth]{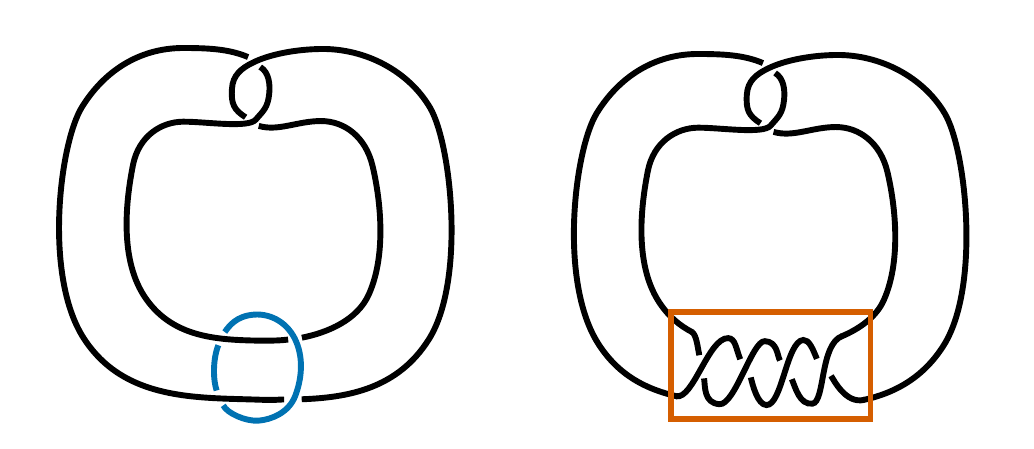}
    \caption{Left: The Whitehead link, with unknotted component to be Dehn filled shown in blue. Right: The twist knot $J(2,-4)$, with 4 left-handed crossings inside the twist region indicated by the orange box.}
    \label{fig:WHlink}
\end{figure}

In this section we apply Theorem~\ref{thm:Apoly_precise} to the family of twist knots obtained by Dehn filling the Whitehead link using layered solid tori. We use the same triangulation as in~\cite{HMP}, which is given in Regina notation in Figure~\ref{fig:WHlinktri}. The Dehn fillings are performed by replacing tetrahedra 3 and 4 with layered solid tori. Each of the three outside tetrahedra contribute a Ptolemy equation, which were found in~\cite{HMP}. We make the substitutions $\ell=L^2$ and $m=M^2$, and multiply through by powers of $L$ and $M$ to remove negative exponents.  We use the same $\gamma$ labels, including $\gamma_{0(23)}$, which is associated to the edge class in the triangulation that contains the $23$ edge of tetrahedron $0$. This labelling reflects the fact that this edge class does not appear in the cusp being filled and therefore does not have a well-defined slope. The equations are
\begin{align}
    -LM\gamma_{0(23)}\gamma_{2/1}-L\gamma_{3/1}\gamma_{1/0}-M^{2}\gamma_{1/0}^2&=0, \label{eq:1WHlink} \\
    -M^2\gamma_{3/1}\gamma_{1/0}-L\gamma_{1/0}^2-M\gamma_{0(23)}\gamma_{2/1}&=0, \label{eq:2WHlink} \\
    \gamma_{1/0}^2-\gamma_{1/0}\gamma_{3/1}-\gamma_{0(23)}^2&=0. \label{eq:3WHlink}
\end{align}

\begin{figure}
\begin{center}
\begin{tabular}{c|c|c|c|c}
    Tetrahedron & Face 012 & Face 013 & Face 023 & Face 123 \\
    \hline
    0 & 3(021) & 1(213) & 2(130) & 1(230) \\
    1 & 4(102) & 2(132) & 0(312) & 0(103) \\
    2 & 2(203) & 0(302) & 2(102) & 1(031) \\
    3 & 0(021) & 4(103) & 4(203) & 4(213) \\
    4 & 1(102) & 3(103) & 3(203) & 3(213) 
\end{tabular}
\end{center}
\caption{A triangulation of the Whitehead link complement in Regina notation.}
\label{fig:WHlinktri}
\end{figure}

The paths in the Farey triangulation corresponding to the $\pm 1/m$ Dehn fillings of the Whitehead link were determined by Howie, Mathews and Purcell in~\cite{HMP} and are shown in Figure~\ref{fig:FareypathsWHlink}.  The words describing the positive and negative paths are LRL$^{m-2}$ and L$^{2}$R$^{m-1}$, respectively. Given that the final L or R corresponds to the tip representing the fold, we have tails of lengths $n=m-3$ and $n=m-2$, respectively. From this we see that Theorem~\ref{thm:Apoly_precise} applies to the calculation of A-polynomials for the twist knots $J(2,2m)$ for $m\geq 4$ and $J(2,-2m)$ for $m\geq 3$. The tails of each of the positive and negative Dehn fillings start at step 3, with both paths sharing steps 0 and 1. The Ptolemy equations for steps 0 and 1 are
\begin{align}
    \gamma_{3/1}\gamma_{1/1}+\gamma_{2/1}^2-\gamma_{1/0}^2&=0, \text{ and} \label{eq:4WHlink} \\
    -\gamma_{2/1}\gamma_{0/1}+\gamma_{1/1}^2-\gamma_{1/0}^2&=0. \label{eq:5WHlink}
\end{align}

\begin{figure}
    \centering
    \includegraphics[width=0.6\textwidth]{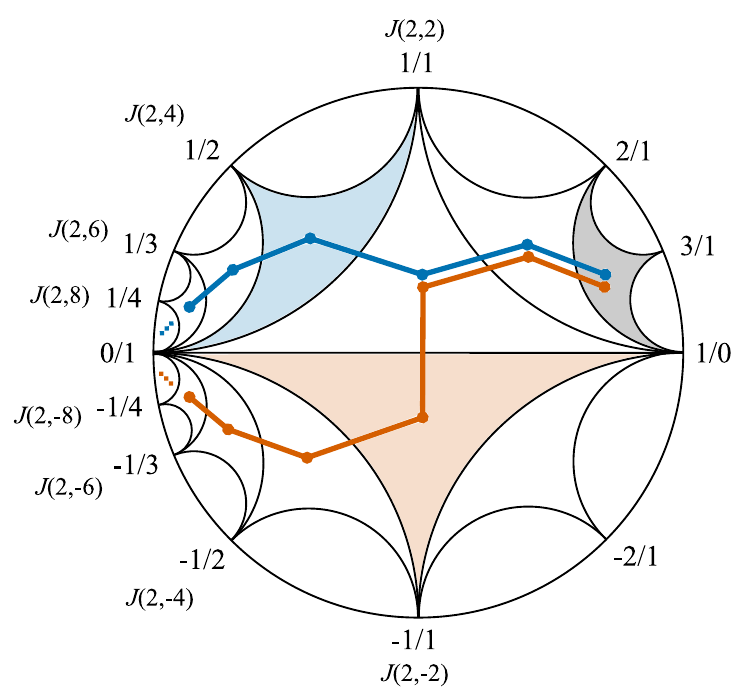}
    \caption{Paths in the Farey diagram corresponding to $\pm 1/m$ Dehn fillings of the Whitehead link. The starting triangle is shaded grey and the triangles where each of the tails begin are indicated in blue and orange, for the positive and negative Dehn fillings, respectively. The twist knot obtained by performing each Dehn filling is also noted.}
    \label{fig:FareypathsWHlink}
\end{figure}

For $+1/m$ Dehn fillings, step 2 is a right step and therefore corresponds to the Ptolemy equation
\begin{equation}
    \gamma_{1/2}\gamma_{1/0}+\gamma_{0/1}^2-\gamma_{1/1}^2=0. \label{eq:step2pos}
\end{equation}
Meanwhile, for $-1/m$ Dehn fillings, step 2 is a left step, so corresponds to the Ptolemy equation 
\begin{equation}
    \gamma_{-1/1}\gamma_{1/1}+\gamma_{0/1}^2-\gamma_{1/0}^2=0. \label{eq:step2neg}
\end{equation}

Next we express $\gamma_{1/1},\gamma_{0/1},\gamma_{1/2},\gamma_{-1/1}$ and $\gamma_{1/0}$ in terms of only $L$ and $M$. We set $\gamma_{1/0}$ equal to 1 and use equations~\eqref{eq:1WHlink} through~\eqref{eq:5WHlink} to express $\gamma_{1/1}$ and $\gamma_{0/1}$ as follows:
\begin{align*}
    \gamma_{1/1}&= M^{-2}(L-1)^{-1}(L-M^4) \\
    \gamma_{0/1}&= -M^{-3}L^{1/2}(L-1)^{-3/2}(M-1)(M+1)(M^2-1)^{1/2}(M^2+1)(L+M^2)^{1/2}.
\end{align*}

For $+1/m$ Dehn fillings, we rearrange \eqref{eq:step2pos} to get
\begin{align*}
    \gamma_{1/2} &= M^{-6}(L-1)^{-3}(L^2 +L M^2- 2L^2 M^2 +L^3 M^2- L M^4 - 2L^2 M^4 +2 L M^8 + L^2 M^8 \\
    & \qquad - M^{10}+ 2L M^{10} - L^2 M^{10} - L M^{12}).
\end{align*}

For $-1/m$ Dehn fillings, we rearrange \eqref{eq:step2neg} to get
\begin{equation*}
    \gamma_{-1/1}= M^{-4}(L-1)^{-2}(L+M^2-L M^2- 2L M^4- L M^6 + L^2 M^6 + L M^8).
\end{equation*}

For the positive Dehn fillings, Theorem~\ref{thm:Apoly_precise} tells us that the A-polynomial for  $J(2,2(n+3))$, for $n\geq1$, contains a factor of the rational function given by 
\begin{equation}
     \gamma_{1/2}^{2n} + \sum_{a+b\leq n-1} (-1)^{n-a-b}\binom{n-1-a}{b}\binom{n-b}{a}\gamma_{1/2}^{2a}\gamma_{1/1}^{2b}\gamma_{0/1}^{2(n-a-b)} - \gamma_{1/1}^{n}\gamma_{1/2}^{n-1}\gamma_{0/1} =0. 
\end{equation}
\begin{remark}
    As stated, the output of this equation involves fractional exponents, however, these can be removed by conjugating appropriately. 
\end{remark}

Meanwhile, for negative Dehn fillings, the A-polynomial for  $J(2,-2(n+2))$ contains a factor of the rational function given by 
\begin{equation}
     \gamma_{1/2}^{2n} + \sum_{a+b\leq n-1} (-1)^{n-a-b}\binom{n-1-a}{b}\binom{n-b}{a}\gamma_{-1/1}^{2a}\gamma_{1/0}^{2b}\gamma_{0/1}^{2(n-a-b)} - \gamma_{1/0}^{n}\gamma_{-1/1}^{n-1}\gamma_{0/1} = 0. 
\end{equation}
Again, conjugation is needed to remove fractional exponents.

\subsubsection{Comparing with what is known}
The A-polynomials for the twist knots were shown by Hoste and Shanahan to be irreducible. As such, we expect the output of our formulas to contain precisely the A-polynomial of each twist knot. Using the explicit formulas of Mathews~\cite{Mathews2014,Mathews2014err}, we verify that the largest factor of our output is indeed the A-polynomial for each of the twist knots $J(2,2m)$ for $m\in\left[-8,-3\right]\cup\left[4,8\right]$. Note that a change of basis is required, namely $(L,M)\mapsto(-LM^{-2},M)$. This change of basis does not depend on the Dehn filling slope, since the linking number of the two components of the Whitehead link is $0$.

The behaviour seen in the output of our formulas uncovers a new recursive relationship in the A-polynomials of twist knots. In the following, we let 
\begin{align*}
        x&= -L +L^2 +2LM^2 +M^4+2LM^4+L^2M^4+2LM^6 +M^8 -LM^8, \\
        y&= M^4 (L+M^2)^4,\text{ and }\\
        z&=L (M^2-1)^3(M^2+1)^2(L-M^4).
    \end{align*}
\begin{theorem}
    Let $A_{m}^+$ be the A-polynomial of the twist knot $J(2,2m)$ and let $A_{m}^-$ be the A-polynomial of the twist knot $J(2,-2m)$. With initial conditions below,
    \begin{align*}
    A_{m-1}^+ \cdot A_{m+1}^+ &= (A_m^+)^2 + y^{m-2}z M^{4}(L+M^2)^{3}, \text{ for } m>1, \text{ and } \\
    A_{m-1}^- \cdot A_{m+1}^- &= (A_m^-)^2 + y^{m-1}z (L+M^2), \text{ for } m>0.
    \end{align*}

Initial conditions:
\begin{align*}
    A_{1}^-&= -L + LM^2 + M^4 +2 LM^4 + L^2 M^4 + LM^6 - LM^8  \\
    A_0^-&=1 \\
    A_{1}^+&= L + M^6 \\
    A_{2}^+&= -L^2 + L^3 + 2L^2M^2 + LM^4 + 2L^2M^4 -LM^6 - L^2M^8 \\
    &\quad +2LM^{10} + L^2M^{10} + 2LM^{12} +M^{14} -LM^{14} 
\end{align*}
\end{theorem} 

\begin{proof}
    We use the recursive relationship found by Hoste and Shanahan~\cite{HosteShanahan04}, which can be rewritten in our notation as follows, with $x,y$, and initial conditions as above.
    \[ A_m^+ = x A_{m-1}^+ -y A_{m-2}^+,   \]
    \[ A_m^- = x A_{m-1}^- -y A_{m-2}^-.   \]
   
    We prove the positive case by showing that 
    \[ A_{m-1}^+ \cdot A_{m+1}^+ - (A_m^+)^2 = y^{m-2}z M^{4}(L+M^2)^{3}. \]
   
    Applying Hoste and Shanahan's relation repeatedly to the left-hand side gives 
    \begin{align*}
        A_{m-1}^+ \cdot A_{m+1}^+ - (A_m^+)^2 &= x A_{m-1}^+ \cdot A_{m}^+ - y (A_{m-1}^+)^2 - (A_{m}^+)^2 \\
        &= y \left(x A_{m-2}^+ \cdot A_{m-1}^+ - y (A_{m-2}^+)^2 - (A_{m-1}^+)^2\right) \\
        & \quad\vdots \\
        &= y^k \left(x A_{m-k-1}^+ \cdot A_{m-k}^+ - y (A_{m-k-1}^+)^2 - (A_{m-k}^+)^2\right) \\
        & \quad\vdots \\
        &= y^{m-2} \left(x A_{1}^+ \cdot A_{2}^+ - y (A_{1}^+)^2 - (A_{2}^+)^2\right).
    \end{align*}
    
    Substituting in the expressions for $x,y,A_{1}^+$ and $A_{2}^+$, we find that 
    \[ x A_{1}^+ \cdot A_{2}^+ - y (A_{1}^+)^2 - (A_{2}^+)^2 = z M^{4}(L+M^2)^{3}, \]
    thus recovering the right-hand side. 
    
    An analogous argument shows that 
    \[ A_{m-1}^- \cdot A_{m+1}^- - (A_m^-)^2 = y^{m-1} \left(x A_{0}^- \cdot A_{1}^- - y (A_{0}^-)^2 - (A_{1}^-)^2\right). \]
    Again by substituting, we find that 
    \[ x A_{0}^- \cdot A_{1}^- - y (A_{0}^-)^2 - (A_{1}^-)^2 = z (L+M^2), \]
    which proves the negative case.
\end{proof}
\vfill

\appendix
\section{\texorpdfstring{Variables in $L$ and $M$ for fillings of the $(-2,3,8)$-pretzel link}{Variables in L and M for fillings of the (-2,3,8)-pretzel link}}\label{app:gammas}
In order to express equations~\eqref{eq:tailpos} and~\eqref{eq:tailneg} entirely in terms of $L$ and $M$ we need the variables $\gamma_{1/0}, \gamma_{1/1}, \gamma_{0/1}, \gamma_{-1/1}$ and $\gamma_{1/2}$ in terms of $L$ and $M$. These are summarised below. With these substitutions, equations~\eqref{eq:tailpos} and~\eqref{eq:tailneg} become formulas for rational functions that contain the geometric factor of the A-polynomial for the twisted torus knots $T(5,1-5m,2,2)$.

\begin{equation*}
    \gamma_{1/0}=M^{-1} (L-M)^{-1} (L+M)^{-1} (L-M^2) (L+M^2) 
\end{equation*}

\begin{align*}
    \gamma_{1/1}&=-M^{-4} (L-M)^{-6} (L+M)^{-6} \left(L^6 M+L^5 M^4-2 L^5 M^2+L^5-L^4 M^5-2 L^4 M^3+2 L^2 M^7 \right. \\
    & \qquad \left. +L^2 M^5-L M^{10}+2 L M^8-L M^6-M^9\right) \left(L^6 M-L^5 M^4+2 L^5 M^2-L^5-L^4 M^5 \right. \\
    & \qquad \left. -2 L^4 M^3+2 L^2 M^7+L^2 M^5+L M^{10}-2 L M^8+L M^6-M^9\right)
\end{align*}

\begin{align*}
    \gamma_{0/1}&= -L^{-1} M^{-6} (M-1)^{-1} (M+1)^{-1} (L-M)^{-9} (L+M)^{-9}\left(L^{10} M^2-L^8 M^7+3 L^8 M^5-7 L^8 M^4\right. \\
    & \qquad -3 L^8 M^3+3 L^8 M^2+L^8 M-L^8-L^6 M^{10}+L^6 M^9+7 L^6 M^8-3 L^6 M^7+3 L^6 M^6 \\
    & \qquad +3 L^6 M^5+L^6 M^4-L^6 M^3+L^4 M^{13}-L^4 M^{12}-3 L^4 M^{11}-3 L^4 M^{10}+3 L^4 M^9-7 L^4 M^8 \\
    & \qquad -L^4 M^7+L^4 M^6+L^2 M^{16}-L^2 M^{15}-3 L^2 M^{14}+3 L^2 M^{13}+7 L^2 M^{12}-3 L^2 M^{11}+L^2 M^9 \\
    & \qquad \left. -M^{14}\right)\left(L^{10} M^2+L^8 M^7-3 L^8 M^5-7 L^8 M^4+3 L^8 M^3+3 L^8 M^2-L^8 M-L^8-L^6 M^{10} \right. \\
    & \qquad -L^6 M^9+7 L^6 M^8+3 L^6 M^7+3 L^6 M^6-3 L^6 M^5+L^6 M^4+L^6 M^3-L^4 M^{13}-L^4 M^{12} \\
    & \qquad +3 L^4 M^{11}-3 L^4 M^{10}-3 L^4 M^9-7 L^4 M^8+L^4 M^7+L^4 M^6+L^2 M^{16}+L^2 M^{15}-3 L^2 M^{14} \\
    & \qquad \left. -3 L^2 M^{13}+7 L^2 M^{12}+3 L^2 M^{11}-L^2 M^9-M^{14}\right)
\end{align*}

\begin{align*}
    \gamma_{-1/1}&= L^{-2} M^{-8} (M-1)^{-2} (M+1)^{-2} (L-M)^{-12} (L+M)^{-12}\left(-L^3 M^{22}+L^2 M^{21}+5 L^3 M^{20}\right. \\
    & \qquad +L M^{20}-2 L^4 M^{19}-3 L^2 M^{19}-M^{19}-14 L^3 M^{18}-L M^{18}+L^6 M^{17}+5 L^4 M^{17} \\
    & \qquad +9 L^2 M^{17}+2 L^7 M^{16}+6 L^5 M^{16}+12 L^3 M^{16}-2 L^6 M^{15}-18 L^4 M^{15}-14 L^7 M^{14} \\
    & \qquad +2 L^5 M^{14}-3 L^3 M^{14}+L^8 M^{13}+20 L^6 M^{13}-7 L^4 M^{13}+L^9 M^{12}+12 L^7 M^{12}-7 L^5 M^{12} \\
    & \qquad +L^3 M^{12}-L^{10} M^{11}-17 L^8 M^{11}+17 L^6 M^{11}+L^4 M^{11}-L^{11} M^{10}+7 L^9 M^{10}-12 L^7 M^{10} \\
    & \qquad -L^5 M^{10}+7 L^{10} M^9-20 L^8 M^9-L^6 M^9+3 L^{11} M^8-2 L^9 M^8+14 L^7 M^8+18 L^{10} M^7 \\
    & \qquad +2 L^8 M^7-12 L^{11} M^6-6 L^9 M^6-2 L^7 M^6-9 L^{12} M^5-5 L^{10} M^5-L^8 M^5+L^{13} M^4 \\
    & \qquad \left. +14 L^{11} M^4+L^{14} M^3+3 L^{12} M^3+2 L^{10} M^3-L^{13} M^2-5 L^{11} M^2-L^{12} M+L^{11}\right) \left(L^3 M^{22}\right. \\
    & \qquad +L^2 M^{21}-5 L^3 M^{20}-L M^{20}-2 L^4 M^{19}-3 L^2 M^{19}-M^{19}+14 L^3 M^{18}+L M^{18}+L^6 M^{17} \\
    & \qquad +5 L^4 M^{17}+9 L^2 M^{17}-2 L^7 M^{16}-6 L^5 M^{16}-12 L^3 M^{16}-2 L^6 M^{15}-18 L^4 M^{15}+14 L^7 M^{14} \\
    & \qquad -2 L^5 M^{14}+3 L^3 M^{14}+L^8 M^{13}+20 L^6 M^{13}-7 L^4 M^{13}-L^9 M^{12}-12 L^7 M^{12}+7 L^5 M^{12} \\
    & \qquad -L^3 M^{12}-L^{10} M^{11}-17 L^8 M^{11}+17 L^6 M^{11}+L^4 M^{11}+L^{11} M^{10}-7 L^9 M^{10}+12 L^7 M^{10} \\
    & \qquad +L^5 M^{10}+7 L^{10} M^9-20 L^8 M^9-L^6 M^9-3 L^{11} M^8+2 L^9 M^8-14 L^7 M^8+18 L^{10} M^7 \\
    & \qquad +2 L^8 M^7+12 L^{11} M^6+6 L^9 M^6+2 L^7 M^6-9 L^{12} M^5-5 L^{10} M^5-L^8 M^5-L^{13} M^4 \\
    & \qquad \left. -14 L^{11} M^4+L^{14} M^3+3 L^{12} M^3+2 L^{10} M^3+L^{13} M^2+5 L^{11} M^2-L^{12} M-L^{11}\right)
\end{align*}

\newpage
\begin{align*}
    \gamma_{1/2} &= L^{-2} M^{-11} (M-1)^{-2} (M+1)^{-2} (L-M)^{-17} (L+M)^{-17}\left(-L^4 M^{30}-L^5 M^{28}+7 L^4 M^{28} \right.\\
    &\qquad +L^3 M^{28}+2 L^2 M^{28}-L^6 M^{26}+4 L^5 M^{26}-28 L^4 M^{26}-3 L^3 M^{26}-6 L^2 M^{26}-L M^{26} \\
    &\qquad -M^{26}+3 L^8 M^{24}+2 L^7 M^{24}+16 L^6 M^{24}-13 L^5 M^{24}+52 L^4 M^{24}+11 L^3 M^{24}+13 L^2 M^{24} \\
    &\qquad +2 L^9 M^{22}-27 L^8 M^{22}-L^7 M^{22}-35 L^6 M^{22}+L^5 M^{22}-64 L^4 M^{22}-2 L^3 M^{22}+4 L^{10} M^{20} \\
    &\qquad -16 L^9 M^{20}+70 L^8 M^{20}+40 L^7 M^{20}+53 L^6 M^{20}-27 L^5 M^{20}-L^4 M^{20}+3 L^3 M^{20} \\
    &\qquad -3 L^{12} M^{18}-2 L^{11} M^{18}-18 L^{10} M^{18}-5 L^9 M^{18}-115 L^8 M^{18}+10 L^7 M^{18}+53 L^6 M^{18} \\
    &\qquad -3 L^5 M^{18}-L^4 M^{18}-L^3 M^{18}-L^{13} M^{16}+23 L^{12} M^{16}+34 L^{11} M^{16}+41 L^{10} M^{16} \\
    &\qquad -80 L^9 M^{16}-25 L^8 M^{16}+53 L^7 M^{16}-3 L^6 M^{16}+3 L^5 M^{16}-3 L^{14} M^{14}+3 L^{13} M^{14} \\
    &\qquad -53 L^{12} M^{14}+25 L^{11} M^{14}+80 L^{10} M^{14}-41 L^9 M^{14}-34 L^8 M^{14}-23 L^7 M^{14}+L^6 M^{14} \\
    &\qquad +L^{16} M^{12}+L^{15} M^{12}+3 L^{14} M^{12}-53 L^{13} M^{12}-10 L^{12} M^{12}+115 L^{11} M^{12}+5 L^{10} M^{12} \\
    &\qquad +18 L^9 M^{12}+2 L^8 M^{12}+3 L^7 M^{12}-3 L^{16} M^{10}+L^{15} M^{10}+27 L^{14} M^{10}-53 L^{13} M^{10} \\
    &\qquad -40 L^{12} M^{10}-70 L^{11} M^{10}+16 L^{10} M^{10}-4 L^9 M^{10}+2 L^{16} M^8+64 L^{15} M^8-L^{14} M^8 \\
    &\qquad +35 L^{13} M^8+L^{12} M^8+27 L^{11} M^8-2 L^{10} M^8-13 L^{17} M^6-11 L^{16} M^6-52 L^{15} M^6 \\
    &\qquad +13 L^{14} M^6-16 L^{13} M^6-2 L^{12} M^6-3 L^{11} M^6+L^{19} M^4+L^{18} M^4+6 L^{17} M^4+3 L^{16} M^4 \\
    &\qquad \left. +28 L^{15} M^4-4 L^{14} M^4+L^{13} M^4-2 L^{17} M^2-L^{16} M^2-7 L^{15} M^2+L^{14} M^2+L^{15}\right)  \\
    &\qquad \left(L^4 M^{30}-L^5 M^{28}-7 L^4 M^{28}+L^3 M^{28}-2 L^2 M^{28}+L^6 M^{26}+4 L^5 M^{26}+28 L^4 M^{26} \right.\\
    &\qquad -3 L^3 M^{26}+6 L^2 M^{26}-L M^{26}+M^{26}-3 L^8 M^{24}+2 L^7 M^{24}-16 L^6 M^{24}-13 L^5 M^{24} \\
    &\qquad -52 L^4 M^{24}+11 L^3 M^{24}-13 L^2 M^{24}+2 L^9 M^{22}+27 L^8 M^{22}-L^7 M^{22}+35 L^6 M^{22} \\
    &\qquad +L^5 M^{22}+64 L^4 M^{22}-2 L^3 M^{22}-4 L^{10} M^{20}-16 L^9 M^{20}-70 L^8 M^{20}+40 L^7 M^{20} \\
    &\qquad -53 L^6 M^{20}-27 L^5 M^{20}+L^4 M^{20}+3 L^3 M^{20}+3 L^{12} M^{18}-2 L^{11} M^{18}+18 L^{10} M^{18} \\
    &\qquad -5 L^9 M^{18}+115 L^8 M^{18}+10 L^7 M^{18}-53 L^6 M^{18}-3 L^5 M^{18}+L^4 M^{18}-L^3 M^{18}-L^{13} M^{16} \\
    &\qquad -23 L^{12} M^{16}+34 L^{11} M^{16}-41 L^{10} M^{16}-80 L^9 M^{16}+25 L^8 M^{16}+53 L^7 M^{16}+3 L^6 M^{16} \\
    &\qquad +3 L^5 M^{16}+3 L^{14} M^{14}+3 L^{13} M^{14}+53 L^{12} M^{14}+25 L^{11} M^{14}-80 L^{10} M^{14}-41 L^9 M^{14} \\
    &\qquad +34 L^8 M^{14}-23 L^7 M^{14}-L^6 M^{14}-L^{16} M^{12}+L^{15} M^{12}-3 L^{14} M^{12}-53 L^{13} M^{12} \\
    &\qquad +10 L^{12} M^{12}+115 L^{11} M^{12}-5 L^{10} M^{12}+18 L^9 M^{12}-2 L^8 M^{12}+3 L^7 M^{12}+3 L^{16} M^{10} \\
    &\qquad +L^{15} M^{10}-27 L^{14} M^{10}-53 L^{13} M^{10}+40 L^{12} M^{10}-70 L^{11} M^{10}-16 L^{10} M^{10}-4 L^9 M^{10} \\
    &\qquad -2 L^{16} M^8+64 L^{15} M^8+L^{14} M^8+35 L^{13} M^8-L^{12} M^8+27 L^{11} M^8+2 L^{10} M^8 \\
    &\qquad -13 L^{17} M^6+11 L^{16} M^6-52 L^{15} M^6-13 L^{14} M^6-16 L^{13} M^6+2 L^{12} M^6-3 L^{11} M^6 \\
    &\qquad +L^{19} M^4-L^{18} M^4+6 L^{17} M^4-3 L^{16} M^4+28 L^{15} M^4+4 L^{14} M^4+L^{13} M^4-2 L^{17} M^2 \\
    &\qquad \left. +L^{16} M^2-7 L^{15} M^2-L^{14} M^2+L^{15}\right)
\end{align*}
\vfill
\newpage
\bibliographystyle{amsplain} 
\bibliography{references}
\end{document}

%% file: WeightedLadderGraph.pdf_tex
\begingroup%
  \makeatletter%
  \providecommand\color[2][]{%
    \errmessage{(Inkscape) Color is used for the text in Inkscape, but the package 'color.sty' is not loaded}%
    \renewcommand\color[2][]{}%
  }%
  \providecommand\transparent[1]{%
    \errmessage{(Inkscape) Transparency is used (non-zero) for the text in Inkscape, but the package 'transparent.sty' is not loaded}%
    \renewcommand\transparent[1]{}%
  }%
  \providecommand\rotatebox[2]{#2}%
  \newcommand*\fsize{\dimexpr\f@size pt\relax}%
  \newcommand*\lineheight[1]{\fontsize{\fsize}{#1\fsize}\selectfont}%
  \ifx\svgwidth\undefined%
    \setlength{\unitlength}{538.58267717bp}%
    \ifx\svgscale\undefined%
      \relax%
    \else%
      \setlength{\unitlength}{\unitlength * \real{\svgscale}}%
    \fi%
  \else%
    \setlength{\unitlength}{\svgwidth}%
  \fi%
  \global\let\svgwidth\undefined%
  \global\let\svgscale\undefined%
  \makeatother%
  \begin{picture}(1,0.22105263)%
    \lineheight{1}%
    \setlength\tabcolsep{0pt}%
    \put(0,0){\includegraphics[width=\unitlength,page=1]{WeightedLadderGraph.pdf}}%
    \put(0.86337714,0.02609648){\makebox(0,0)[lt]{\lineheight{1.25}\smash{\begin{tabular}[t]{l}$\gamma_{f_{k}}$\end{tabular}}}}%
    \put(0.73108545,0.02609648){\makebox(0,0)[lt]{\lineheight{1.25}\smash{\begin{tabular}[t]{l}$\gamma_{o_{k}}$\end{tabular}}}}%
    \put(0.59879391,0.02609648){\makebox(0,0)[lt]{\lineheight{1.25}\smash{\begin{tabular}[t]{l}$\gamma_{f_{k}}$\end{tabular}}}}%
    \put(0.34117319,0.02609648){\makebox(0,0)[lt]{\lineheight{1.25}\smash{\begin{tabular}[t]{l}$\gamma_{f_{k}}$\end{tabular}}}}%
    \put(0.20888152,0.02609648){\makebox(0,0)[lt]{\lineheight{1.25}\smash{\begin{tabular}[t]{l}$\gamma_{o_{k}}$\end{tabular}}}}%
    \put(0.0765899,0.02609648){\makebox(0,0)[lt]{\lineheight{1.25}\smash{\begin{tabular}[t]{l}$\gamma_{f_{k}}$\end{tabular}}}}%
    \put(0.8216009,0.10964913){\makebox(0,0)[lt]{\lineheight{1.25}\smash{\begin{tabular}[t]{l}$\gamma_{p_{k}}$\end{tabular}}}}%
    \put(0.9538926,0.10964913){\makebox(0,0)[lt]{\lineheight{1.25}\smash{\begin{tabular}[t]{l}$-\gamma_{p_{k}}$\end{tabular}}}}%
    \put(0.68930929,0.10964913){\makebox(0,0)[lt]{\lineheight{1.25}\smash{\begin{tabular}[t]{l}$-\gamma_{p_{k}}$\end{tabular}}}}%
    \put(0.5570176,0.10964913){\makebox(0,0)[lt]{\lineheight{1.25}\smash{\begin{tabular}[t]{l}$\gamma_{p_{k}}$\end{tabular}}}}%
    \put(0.42472586,0.10964913){\makebox(0,0)[lt]{\lineheight{1.25}\smash{\begin{tabular}[t]{l}$-\gamma_{p_{k}}$\end{tabular}}}}%
    \put(0.29243419,0.10964913){\makebox(0,0)[lt]{\lineheight{1.25}\smash{\begin{tabular}[t]{l}$\gamma_{p_{k}}$\end{tabular}}}}%
    \put(0.16014253,0.10964913){\makebox(0,0)[lt]{\lineheight{1.25}\smash{\begin{tabular}[t]{l}$-\gamma_{p_{k}}$\end{tabular}}}}%
    \put(0.02785088,0.10964913){\makebox(0,0)[lt]{\lineheight{1.25}\smash{\begin{tabular}[t]{l}$\gamma_{p_{k}}$\end{tabular}}}}%
    \put(0.86337722,0.19320176){\makebox(0,0)[lt]{\lineheight{1.25}\smash{\begin{tabular}[t]{l}$\gamma_{f_{k}}$\end{tabular}}}}%
    \put(0.73108545,0.19320176){\makebox(0,0)[lt]{\lineheight{1.25}\smash{\begin{tabular}[t]{l}$\gamma_{o_{k}}$\end{tabular}}}}%
    \put(0.59879399,0.19320176){\makebox(0,0)[lt]{\lineheight{1.25}\smash{\begin{tabular}[t]{l}$\gamma_{f_{k}}$\end{tabular}}}}%
    \put(0.34117323,0.19320176){\makebox(0,0)[lt]{\lineheight{1.25}\smash{\begin{tabular}[t]{l}$\gamma_{f_{k}}$\end{tabular}}}}%
    \put(0.20888158,0.19320176){\makebox(0,0)[lt]{\lineheight{1.25}\smash{\begin{tabular}[t]{l}$\gamma_{o_{k}}$\end{tabular}}}}%
    \put(0.07658995,0.19320176){\makebox(0,0)[lt]{\lineheight{1.25}\smash{\begin{tabular}[t]{l}$\gamma_{f_{k}}$\end{tabular}}}}%
  \end{picture}%
\endgroup%